\documentclass[a4paper,12pt,reqno]{article}
 \usepackage[english]{babel}
  \usepackage{amsmath,amsfonts,amssymb,amsthm,bbm} 
   \usepackage{graphics,epsfig,psfrag} %%add this and next lines if pictures
   \usepackage{paralist,subfigure}
   \usepackage{amsmath}
   \usepackage[dvipsnames]{xcolor}
   \usepackage{hyperref} %%Warning: when you first run your tex
    \usepackage[active]{srcltx} %%allows you to jump from your editor
    \usepackage{color,soul} %%allows you to use the \st{...} command, which
    \usepackage[latin1]{inputenc}
    \usepackage[OT1]{fontenc}
    \usepackage{authblk}

\providecommand{\Real}{\mathop{\rm Re}\nolimits}%
\providecommand{\Imag}{\mathop{\rm Im}\nolimits}%

\begin{document}

\theoremstyle{plain}

\newtheorem{theorem}{Theorem}
\newtheorem{lemma}{Lemma}
\newtheorem{cor}{Corollary}
\newtheorem{prop}{Proposition}
\newtheorem{req}{Remark}
\newtheorem{definition}{Definition}
\newtheorem{hyp}{Hypothesis}

\setlength{\marginparwidth}{1in}

\newenvironment{Romain}
    {\begin{center}
     \begin{tabular}{|p{0.95\textwidth}|}\hline\color{ForestGreen}}
    { \vspace{5pt}\color{black}\\\hline\end{tabular} 
    \end{center}}

\title{\bf Influence of the geometry on a field-road model : the case of a conical field}
\makeatletter{\renewcommand*{\@makefnmark}{}\footnotetext{email : romain.ducasse@dauphine.fr}\makeatother}

\author[*]{Romain {\sc Ducasse}}

\affil[*]{Ecole des Hautes Etudes en Sciences Sociales, PSL Research University,  
		Centre d'Analyse et Math\'ematiques Sociales, 54 boulevard Raspail 
		75006 Paris, France}
		
\date{}

\maketitle


\begin{abstract}
Field-road models are reaction-diffusion systems which have been recently introduced to account for the effect of a road on propagation phenomena arising in epidemiology and ecology. Such systems consist in coupling a classical Fisher-KPP equation to a line with fast diffusion accounting for a road. A series of works investigate the spreading properties of such systems when the road is a straight line and the field a half-plane. Here, we take interest in the case where the field is a cone. Our main result is that the spreading speed is not influenced by the angle of the cone.
\end{abstract}

\section{Introduction, known results and presentation of the model}

\subsection{Field-road models}
The study of reaction-diffusion equations and systems is motivated by a wide range of applications, in particular in ecology and epidemiology. Such equations can indeed model propagation phenomena arising from the combined effects of diffusion, which accounts in a biological context for random motion of individuals, and reaction (resulting from reproduction and mortality). The most iconic example is the Fisher-KPP equation (see \cite{F} for the biological motivations and \cite{KPP} for a mathematical study):
\[
\partial_{t}u -d\Delta u = f(u) \, , \, t>0 \, ,  x \in \mathbb{R}^{N}.
\]
The function $f$ is supposed to satisfy the KPP hypothesis, i.e., $f$ is a locally Lipschitz continuous function on $\left[ 0, 1 \right]$, $f>0$ on $]0,1[$, $f(0)=f(1)=0$ and $v \mapsto \frac{f(v)}{v}$ is non-increasing (hence, $f$ is differentiable at $0$ and $f^{\prime}(0)>0$). This is called the {\em KPP property}, and this setting will be assumed through all the paper. A typical example is the logistic non-linearity $f(v)=v(1-v)$. In the KPP setting, it is proven in \cite{KPP} (see also \cite{AW} for further properties) that, if $u_{0}$ is a non-negative compactly supported initial datum not everywhere equal to zero, then the solution $u(t,x)$ of the Fisher-KPP equation arising from this datum goes to $1$ as $t$ goes to infinity, locally uniformly in $x$. This property is called  {\em invasion}. One can then define the speed of invasion as a quantity $c_{KPP} \geq 0$ such that:
\[
\begin{array}{l}
\forall \, c > c_{KPP} \, , \, \displaystyle{\sup_{\vert x \vert \geq ct}} u(t,x) \to 0  \, \text{ as} \,\, t \to +\infty\\
\forall \, c < c_{KPP} \, , \, \displaystyle{\inf_{\vert x \vert \leq ct  }} u(t,x) \to 1 \, \text{ as} \,\, t \to +\infty.
\end{array}
\]
Moreover, the speed of invasion can be explicitly computed in this case : $c_{KPP}=2\sqrt{d f^{\prime}(0)}$, see \cite{AW}.

In \cite{BRR1}, Berestycki, Roquejoffre and Rossi proposed a new system specifically devised for modeling the role of roads in biological invasions (informally, a road is a region where the density $u$ diffuses very fast but where no reproduction takes place). Their model is a system of two coupled reaction-diffusion equations. They consider a population diffusing in a field (i.e., a half-plane) with a diffusivity constant $d$ and reacting, or reproducing, with a KPP reaction term $f$. We call $v(t,x,y)$ the density of population at time $t$ at point $(x,y)$. On the boundary of the field, there is a road, i.e., a line where the population diffuses at a different rate (i.e., with a diffusivity constant $D$ different from $d$) but where there is no reaction term. The density of the population on the road is denoted $u(t,x)$. Moreover, we assume that the population on the field can leave the field for the road with some probability, and that the population on the road can enter the field with some (a priori different) probability. The system will then read:
\begin{equation*}
\left\{
\begin{array}{llcl}
\partial_{t}u(t,x)-D\partial_{xx}^{2}u(t,x) &= \nu v(t,x,0)-\mu u(t,x) &\text{ for }& \quad t >0 \, , \, x\in \mathbb{R} \\
\partial_{t}v(t,x,y)-d\Delta v(t,x,y) &= f(v) &\text{ for }& \quad t >0 \, , \, (x,y)\in \mathbb{R}\times \mathbb{R}^{+} \\
-d\partial_{y}v(t,x,0) &= \mu u(t,x)-\nu v(t,x,0) &\text{ for }& \quad t >0 \, , \, x\in \mathbb{R},
\end{array}
\right.
\end{equation*}
where $\mu$ and $\nu$ are two positive constants representing the exchanges between the road and the field. It is proven in \cite{BRR1} that, for a given compactly supported non-negative initial datum $(u_{0},v_{0})$ not everywhere equal to zero, spreading in the direction of the road occurs at some speed $c_{BRR}>0$, depending on the parameters $d,D,\mu,\nu$. In particular, when $D>2d$, then $ c_{BRR}>c_{KPP}:=2\sqrt{df^{\prime}(0)}$. This means that the road enhances the speed of invasion in the field. The influence of drift terms and mortality on the road is studied in \cite{BRR2}, and also the spreading in all directions in the field is studied in the paper \cite{BRR3}.

 The model introduced in \cite{BRR1} has been extended in several directions. For example, the case where the exchanges coefficients $\mu, \nu $ are not constant but periodic in $x$ is treated in \cite{GMZ}, non-local exchanges are studied in \cite{P2,P1}, a combustion nonlinearity instead of a KPP one is considered in \cite{D1}. The effect of non-local diffusion is studied in \cite{BCRR_sem,BCRR}. The case where the field is a cylinder with its boundary playing the role of the road is treated in \cite{RTV}.

\subsection{Models with general fields and main results}

In this paper, we study the case where the field is not a half-plane anymore. We focus on the case where the field is asymptotically a cone (in a sense to be made precise later in Hypothesis \ref{acf}), however some of our results hold in a more general case.

The more natural way to define a field-road model with a field that is not a half-plane would be as follow : let $\Omega$ be an open connected set. $\Omega$ will be the field, and the road will be its boundary $\partial \Omega$  (we detail the hypotheses we make on $\Omega$ after). 

In the following, the function $v$ represents the density of population in the field, and $u$ the density on the road. On the field, $v$ diffuses and reacts according to a KPP nonlinearity, as in \cite{BRR1,BRR2,BRR3,GMZ,RTV}. On the road, $u$ is only subject to diffusion. The road being a one-dimensional manifold, this results in the presence of a Laplace-Beltrami operator. Exchanges between $u$ and $v$ occur on $\partial \Omega$ : a proportion $\nu$ of the population on the field enters the road and a proportion $\mu$ of the population on the road enters the field. The system then reads, for a given field $\Omega$ satisfying some hypotheses to be detailed after:
\begin{equation}
\left\{
\begin{array}{llcl}\label{syst}
\partial_{t}u-D\partial_{ss}^{2}u &= \nu v-\mu u &\text{ for }& \quad t >0 \, , \, (x,y)\in \partial \Omega \\
\partial_{t}v-d\Delta v &= f(v) &\text{ for }& \quad t >0 \, , \, (x,y)\in  \Omega \\
d\partial_{n}v &= \mu u-\nu v &\text{ for }& \quad t >0 \, , \, (x,y)\in \partial \Omega,
\end{array}
\right.
\end{equation}
where $\partial^{2}_{ss}$ stands here for Laplace-Beltrami operator on $\partial \Omega$, and $n$ is the outer normal to the boundary.

The main goal of this paper is to study the speed of spreading in the direction of the road when the field is \emph{asymptotically conical}:
\begin{hyp}\label{acf}{Asymptotically conical field.}
Let $\Omega$ be an open connected set. $\Omega$ is an asymptotically conical field if there is a function $\rho$ and a real $a\in \mathbb{R}$ such that $\rho \in C^{2}(\mathbb{R})$, $\rho^{\prime} \in C^{1,\alpha}_{Unif}$, for some $\alpha>0$\footnote{ By $f \in C^{1,\alpha}_{Unif}$, we mean that $f$ is bounded in the $C^{1,\alpha}(\mathbb{R})$ norm $\| f \|_{C^{1,\alpha}}=\|f \|_{L^{\infty}}+\| f^{\prime} \|_{L^{\infty}}+\big[ f^{\prime} \big]_{C^{0,\alpha}}$.} and:

\begin{equation}
\left\{
\begin{array}{lrr}
\vert \rho(x)-a \vert x \vert \vert &\underset{x \to \pm \infty}{\longrightarrow}& 0 \\
\rho^{\prime}(x) &\underset{x \to \pm \infty}{\longrightarrow}& \pm a \\
\rho^{\prime \prime}(x) &\underset{x \to \pm \infty}{\longrightarrow}& 0,
\end{array}
\right.
\end{equation}
and
\[
\Omega = \{ (x,y) \, , \, y \geq \rho(x) , \, x \in \mathbb{R}\}.
\]
We call $\theta_{0}$ the half-angle of opening of the cone, i.e., $2\theta_{0} \in ]0,2\pi]$, where $\theta_{0}=\arctan(\frac{1}{a})+\pi 1_{a<0} \geq 0$ if $a\neq 0$ (if $a=0$, we set $\theta_{0}=\frac{\pi}{2}$). The right (resp. left) direction of the road will be the direction of the vector $(1,a)$ (resp. $ (-1,a)$ ).

\end{hyp}
This for instance includes roads given by graphs of functions as $\rho(x)=\sqrt{1+x^{2}}$. Let us remark that we have the following formula, in local coordinates, for the Laplace-Beltrami operator, if $u\in C^{2}(\partial \Omega)$:
\[
\partial^{2}_{ss}u(x,\rho(x))=\frac{1}{\sqrt{1+(\rho^{\prime})^{2}(x)}}\partial_{x}\Big(\frac{1}{\sqrt{1+(\rho^{\prime})^{2}(x)}}\partial_{x}\big(u(x,\rho(x))\big)\Big).
\]

Our main result is that the speed of invasion in the direction of a branch of the road is not influenced by the other branch of the road, whatever the angle between the roads. There is no acceleration if the angle is small, and no deceleration if the angle is large. Let us state here our main theorem:
 \begin{theorem}\label{mainth}
Consider the field-road system \eqref{syst} with $\Omega$ being an asymptotically conical field, i.e., satisfying Hypothesis \ref{acf}. Let $(u,v)$ be the solution of this system arising from a non-negative and compactly supported initial datum not everywhere equal to zero. Then, the spreading speed in the direction of the road is $c_{BRR}$, in the following sense:

\[
\forall h>0 , \, \forall  c<c_{BRR},
\inf_{\substack{(x,y) \in \Omega \\ dist((x,y),\partial \Omega)<h \\ \vert(x,y)\vert \leq ct}} v(t,x,y) \underset{t\to +\infty}{\longrightarrow}  1 , \inf_{\substack{(x,y) \in \partial\Omega \\ \vert(x,y)\vert \leq ct}} u(t,x,y) \underset{t\to +\infty}{\longrightarrow}  1
\]
and
\[
\forall c>c_{BRR},
\sup_{\substack{(x,y) \in \Omega \\  \vert(x,y)\vert \geq ct}} v(t,x,y) \underset{t\to +\infty}{\longrightarrow}   0 , \sup_{\substack{(x,y) \in \partial\Omega \\ \vert(x,y)\vert \geq ct}} u(t,x,y) \underset{t\to +\infty}{\longrightarrow}  0.
\]

\end{theorem}
For clarity, we study first \emph{exactly conical} fields in Section \ref{conical}. An exactly conical field is an asymptotically conical field with $\rho(x)=a\vert x \vert$ for $\vert x \vert \geq1$.

The difficulty in the case of an exactly conical field is to find supersolutions. In contrast with all the other articles mentioned above, we use almost-radial generalized supersolutions, this is done in subsection \ref{supersol}. In the case of an asymptotically conical field, the difficulty is to build subsolutions (the supersolutions we would have introduced before in the exactly conical case will adapt). This is done in Section \ref{almostconical}.

However, before studying asymptotically conical field, we shall state some (generalized) comparison principle in Section \ref{comp} and a Liouville-type result in Section \ref{LiouvilleSection}. Such results are derived in \cite{BRR1}, the differences here being in dealing with the geometry of the domain. These results hold in the more general case of what we will call a \emph{general field}:
\begin{hyp}\label{gf}{General field.}
Let $\Omega$ be an open connected set. $\Omega$ is a general field if it satisfies the two following hypotheses:

Regularity hypothesis : we require $\partial \Omega$ to be uniformly $C^{2,\alpha}$ for some $\alpha>0$, i.e., it admits a uniformly $C^{2,\alpha}$ atlas.

Geometric hypothesis : there is a smooth open set $S$ such that $\lambda_{S}$, the principal eigenvalue of $-\Delta$ with Dirichlet condition on $S$ satisfies $\lambda_{S} \leq \frac{f^{\prime}(0)}{d}$. Moreover, there is some point $P\in \overline{S}$ such that, for $x \in \Omega$, there is an isometry $\mathbb{I}$ of $\mathbb{R}^{N}$ such that $\mathbb{I}(S)\subset \Omega$ and $\mathbb{I}(P)=x$.

\end{hyp}

Let us make some remarks about these hypotheses. First, the regularity hypothesis implies that $\Omega$ satisfies the uniform interior ball condition~: there is some $r>0$ such that $ \forall x\in \partial \Omega$, $\exists y\in \Omega$ such that $B(y,r) \subset \Omega$ and $x\in \overline{B(y,r)}$.

Secondly, the geometric hypothesis is actually an hypothesis on the size of the domain. Indeed, the principal eigenvalue of $-\Delta$ on a smooth bounded domain $\Omega$ decreases as $\Omega$ increases. This hypothesis simply means that we can slide a set $S$, that is ``large enough", in $\Omega$. If the field is a half-plane, or more generally if it is the epigraph of a function with bounded oscillations (as in the case of Hypothesis \ref{acf}), this hypothesis is automatically satisfied. It is clear that an asymptotically conical field, i.e., satisfying Hypothesis \ref{acf}, satisfies Hypothesis \ref{gf}. This geometric hypothesis prevents the field to collapse at infinity for example.

Our problem relates to those studied in \cite{BRR1,BRR2,BRR3,GMZ,RTV}. The novelty here is the shape of the domains. When studying reaction-diffusion equations, the geometry of the domain can rise serious difficulties. For example, the usual way of proving upper bounds on spreading speeds is to exhibit supersolutions, usually planar. This is not possible anymore in a general domain. Some papers investigate the speed of spreading for reaction-diffusion equations of the KPP type in general domains, see \cite{BHN1,BHN2}.

\section{Preliminary results}\label{comp}

\subsection{Existence and mass conservation}

The existence of solutions of the system \eqref{syst} with a field $\Omega$ satisfying Hypothesis \ref{gf} is classical and will not be treated in details here. Let us just give the main ideas. Consider an initial datum $(U_{0},V_{0})$, bounded and locally Hölder continuous . First, we can define by recurrence $(u_{n},v_{n})$, with $(u_{0},v_{0})=(0,0)$, solution of the two parabolic problems:

\begin{equation*}
\left\{
\begin{array}{llcll}
\partial_{t}u_{n}-D\partial_{ss}^{2}u_{n} + \mu u_{n}&= \nu v_{n-1} &\text{ for }& \quad t >0 \, , \, &(x,y)\in \partial \Omega \\
u_{n}(0,x,y)&=U_{0}(x,y) &\text{ for }& \quad    &(x,y)\in \partial \Omega
\end{array}
\right.
\end{equation*}

\begin{equation*}
\left\{
\begin{array}{llcll}
\partial_{t}v_{n}-d\Delta v_{n} &= f(v_{n}) &\text{ for }& \quad t >0 \, , \, &(x,y)\in  \Omega \\
d\partial_{n}v_{n} +\nu v_{n}&= \mu u_{n} &\text{ for }& \quad t >0 \, , \, &(x,y)\in \partial \Omega \\
v_{n}(0,x,y)&=V_{0}(x,y) & & &(x,y)\in  \Omega
\end{array}
\right.
\end{equation*}

These problems are uniquely solvable, among solutions whose growth is at most $e^{a\vert (x,y)\vert^{2}}$,~ for some $a>0$. Thanks to the parabolic maximum principle, we can get $L^{\infty}$ bounds on the solutions. Then, using $W^{1,2}_{p}$ estimates and Schauder estimates, one can prove that $(u_{n},v_{n})$ converge locally $C^{1,2}_{\alpha}$. Observing that these sequences are non-decreasing, we can pass to the limit to get a solution for the system \eqref{syst}. For the details, see \cite{BRR1,BRR2}, whose proofs adapt easily.

Our system \eqref{syst} has also the following mass conservation property:

\begin{prop}
Let $(u,v)$ be a solution of the system \eqref{syst} with $f=0$ on a general field $\Omega$ satisfying Hypothesis \ref{gf}, arising from the initial datum $(u_{0},v_{0})$ compactly supported. Then:
\[
\| u(t) \|_{L^{1}(\partial\Omega)} + \| v(t) \|_{L^{1}(\Omega)}=\| u(0) \|_{L^{1}(\partial\Omega)} + \| v(0) \|_{L^{1}(\Omega)}.
\]
\end{prop}
See \cite{BRR1} for the proof. The geometry of the domain rises no difficulties here.

Let us now turn to the proof of some comparison principles.

\subsection{Comparison principles}

We prove in this section two comparison principles in the framework of a general field $\Omega$ satisfying Hypothesis \ref{gf}. For the first one, the proof follows the same ideas than in \cite{BRR1}, but we will need to deal with the general geometry of the domain. A key tool will be the following result of Li and Nirenberg \cite{LN} : let $\Omega$ be a domain with $\partial \Omega $ of class $C^{2,\alpha}$ and let $G$ be the largest open subset of $\Omega$ such that for every $x$ in $G$ there is a unique closest point on $\partial \Omega$ to $x$. Then the function distance to $\partial \Omega$ is in $C^{2,\alpha}(G \cup \partial \Omega)$.

In the following, the notation $(u,v)\leq(\overline{u},\overline{v})$ (resp.$ (u,v) < (\overline{u},\overline{v})$  ) will be understood componentwise, i.e., $u\leq \overline{u}$ and $v\leq \overline{v}$ (resp.  $u < \overline{u}$ and $v < \overline{v}$). A subsolution (resp. supersolution) of the system \eqref{syst} is a couple satisfying in the classical sense the system with the signs $=$ replaced by $\leq$ (resp. $\geq$) which are in addition continuous up to time $t=0$.

Let us now state and prove our first comparison principle.

\begin{prop}\label{comparison}
Let $\Omega$ be a general field, i.e., satisfying Hypothesis \ref{gf}. Let $(\underline{u},\underline{v})$ and $(\overline{u},\overline{v})$ be respectively a subsolution bounded from above and a supersolution bounded from below of the field-road system \eqref{syst} with field $\Omega$.

Assume that, at $t=0$, these functions are ordered, i.e.,  $(\underline{u}(0,\cdot),\underline{v}(0,\cdot)) \leq (\overline{u}(0,\cdot),\overline{v}(0,\cdot))$. 
Then, either $(\underline{u},\underline{v}) < (\overline{u},\overline{v})$ for all $t>0$ or there is some $T>0$ such that  $(\underline{u},\underline{v}) = (\overline{u},\overline{v})$ for $t \leq T$.
\end{prop}

\begin{proof}
First, multiplying each of the functions $\underline{u},\underline{v},\overline{u},\overline{v}$ by $e^{-lt}$, $l$ being the Lipschitz constant of $f$, we get subsolutions and supersolutions of the following problem (we still call these functions $ \underline{u}, \underline{v}, \overline{u}, \overline{v}$):

\begin{equation}\label{systcomp}
\left\{
\begin{array}{llcl}
\partial_{t}u-D\partial_{ss}^{2}u + (\mu +l) u &= \nu v &\text{ for } \quad &t >0 \, , \, (x,y)\in \partial \Omega \\
\partial_{t}v-d\Delta v &= h(t,v) &\text{ for } \quad &t >0 \, , \, (x,y)\in  \Omega \\
d\partial_{n}v+\nu v &= \mu u &\text{ for } \quad &t >0 \, , \, (x,y)\in \partial \Omega,
\end{array}
\right.
\end{equation}
where $h(t,v) :=e^{-lt}f(ve^{lt})-lv$ is decreasing with respect to $v$.

Consider the function, defined for $(x,y) \in \Omega$,
\[
\Lambda_{2}(x,y):=\sqrt{1+x^{2}+y^{2}}-2\delta^{\star}(x,y),
\]
where $\delta^{\star}$ is a regularized distance to the boundary built as folllow : let $G$ be the largest open subset of $\Omega$ such that for every $x$ in $G$ there is a unique closest point from $\partial \Omega$ to $x$.
Because $\Omega$ is $C^{2,\alpha}$ uniformly, it satisfies the uniform interior ball condition for some radius $r$. Therefore, it is easy to see that $E:=\{ (x,y)\in \overline{\Omega} \, , \, dist((x,y),\partial \Omega) \leq \frac{r}{2}\} \subset G$. The Li-Nirenberg result ensures that the distance to the boundary is $C^{2,\alpha}$ on $E$. We can set $\delta^{\star}(x,y)=dist((x,y),\partial \Omega)$ if $dist((x,y),\partial \Omega) \leq \frac{r}{4}$, $\delta^{\star}(x,y)=0$ if $dist((x,y),\partial \Omega)\geq \frac{r}{3}$ and $\delta^{\star} \in C^{2,\alpha}(\Omega)$. To do so, one can take $\delta^{\star}(x,y):= \phi(dist((x,y),\partial \Omega))$, where $\phi \in C^{\infty}(\mathbb{R}^{+})$ is a real function such that $\phi(x)=x$ if $0\leq x\leq \frac{r}{4}$ and $\phi(x)=0$ if $x\geq \frac{r}{3}$. Now, let:

\[
\Lambda_{1}:=\Lambda_{2}\vert_{\partial \Omega}.
\]
These functions satisfy:
\begin{equation*}
\left\{
\begin{array}{ccl}
\partial^{2}_{ss}\Lambda_{1}  \leq K  &\text{on} \quad &\partial \Omega \\
\Delta\Lambda_{2}  \leq K &\text{on} \quad &\Omega \\
\partial_{n}\Lambda_{2} \geq 0 &\text{on} \quad &\partial \Omega,
\end{array}
\right.
\end{equation*}
where $K$ is a positive constant, depending only on $\Omega$. The last inequality follows from the fact that $\partial_{n}\delta^{\star}=-1$ on $\partial \Omega$, the first inequality comes from the fact that $ \partial \Omega$ is uniformly $C^{2,\alpha}$.

Let us now consider, for $\epsilon>0$,
\[
\begin{array}{llcc}
\hat{u} := \overline{u}+\epsilon(\Lambda_{1}+t+1) \\
\hat{v} := \overline{v}+\frac{\mu}{\nu}\epsilon(\Lambda_{2}+t+1).
\end{array}
\]
Now, multiplying $\Lambda_{1},\Lambda_{2}$ by some small enough $\eta>0$, it is easy to check that $(\hat{u},\hat{v})$ is still supersolution of \eqref{systcomp}. Moreover, at time $t=0$, $\underline{u}<\hat{u}$ and $\underline{v}<\hat{v}$. By continuity and using the fact that the functions $\hat{u},\hat{v}$ go to infinity at infinity and $\underline{u},\underline{v}$ are bounded from above,  we see that this inequality is still true for $t>0$ sufficiently small. Let us define: 

\[
T := \sup\{ t>0 \, : \, \underline{u} \leq \hat{u} \quad \text{on} \quad (0,t)\times \partial{\Omega} \quad \text{and} \quad \underline{v} \leq \hat{v} \quad \text{on} \quad (0,t)\times \Omega \}.
\]
We have $T>0$. If $T=+\infty$, then we are done. Argue by contradiction and assume that $T< +\infty $. Because the subsolutions are bounded from above, we see that at time $T$, we have contact at some point. Assume first that the contact occurs for $\underline{u},\hat{u}$ i.e.,

\[
\min_{\partial \Omega} (\hat{u}-\underline{u})(T,\cdot)=0.
\]
We know that $(\hat{u},\hat{v})$ is a supersolution and that $(\underline{u},\underline{v}) \leq (\hat{u},\hat{v})$ if $t \leq T$. Therefore, for all such $t \leq T$:

\[
 \partial_{t}\hat{u}-D\partial_{ss}^{2}\hat{u} + (\mu +l) \hat{u} \geq \nu \hat{v}\geq \nu \underline{v} \geq  \partial_{t}\underline{u}-D\partial_{ss}^{2}\underline{u} + (\mu +l) \underline{u}.
\]
We can then apply the strong maximum principle on $\partial \Omega$ to infer that $\underline{u}=\hat{u}$ for all $t \leq T$, which is impossible. Then, we are now left to assume that:

\[
\min_{ \overline{\Omega}} (\hat{v}-\underline{v})( T, \cdot)=0.
\]
If the minimum is reached in the interior of $\Omega$, the usual maximum principle tells us that these two functions are equal for $t \leq T$, which is impossible. Therefore, the minimum must be reached on $\partial \Omega$ and then
\[
0 \geq d\partial_{n}(\hat{v}-\underline{v})+\nu (\hat{v}-\underline{v}) \geq \mu (\hat{u}-\underline{u}) >0.
\]
The last  inequality is strict because the equality is ruled out by what precedes. This is again a contradiction, hence $T=+\infty$.
Therefore, letting $ \epsilon \to 0 $ we have that $ \underline{u} \leq \overline{u}$ and $\underline{v}  \leq \overline{v}  $ for all $t>0$. If we have contact at some time $T>0$, we could use the same arguments again and apply the strong maximum principle to get that $(\underline{u},\underline{v}) = (\overline{u},\overline{v})$ for $t \leq T$.
\end{proof}

Now, let us show what will be called in the following the \emph{generalized comparison theorem}. Indeed, we will often need to deal with generalized subsolutions and supersolutions. 
\begin{prop}\label{gen}
Let $\Omega$ be a general field, i.e., satisfying Hypothesis \ref{gf}. Let $E \subset (0,\infty)\times\partial \Omega$ and $F\subset  (0,\infty)\times\Omega $ be two open sets. Let $(u_{1},v_{1})$ be a subsolution of the field-road system \eqref{syst} with field $\Omega$, only for $(t,x,y) \in E$ in the first equation, for $(t,x,y) \in F$ in the second equation and for $(t,x,y) \in \partial F$  in the third equation. Let $(u_{2},v_{2})$ be  a subsolution of \eqref{syst}. Assume that $(u_{1},v_{1})$ and $(u_{2},v_{2})$ are bounded from above and satisfy:
\[
u_{1} \leq u_{2} \, \text{on} \, (\partial E)\cap((0,\infty)\times\partial \Omega), \quad \quad v_{1} \leq v_{2} \, \text{on} \, (\partial F)\cap((0,\infty)\times \Omega).
\]
Define now $\underline{u}$, $\underline{v}$ by:
\begin{equation*}
\begin{array}{l}
 \underline{u}(t,x,y) := \left\{
				\begin{array}{ll}
				max(u_{1}(t,x,y),u_{2}(t,x,y)) \, &\text{ if } \, (t,x,y) \in \overline{E} \\
				u_{2}(t,x,y) \, &\text{ otherwise}\\
				\end{array}
				\right.
				\\
				\\
\underline{v}(t,x,y) := \left\{
				\begin{array}{ll}
				max(v_{1}(t,x,y),v_{2}(t,x,y)) \, &\text{ if } \, (t,x,y) \in \overline{F} \\
				v_{2}(t,x) \, &\text{ otherwise.}\\
				\end{array}
				\right.
\end{array}
\end{equation*}
If these functions satisfy:
\[
\begin{array}{l}
\underline{u}(t,x,y) > u_{2}(t,x,y) \implies \underline{v}(t,x,y) \geq v_{1}(t,x,y) \, \text{ for } (x,y) \in \partial{\Omega} \\
\underline{v}(t,x,y) > v_{2}(t,x,y)  \implies \underline{u}(t,x,y) \geq u_{1}(t,x,y) \, \text{ for } (x,y) \in \partial{\Omega} ,
\end{array}
\]
then, any supersolution $(\overline{u},\overline{v})$ of \eqref{syst} bounded from below and such that $\underline{u} \leq \overline{u}$ and $\underline{v} \leq \overline{v}$ at $t=0$ satisfies  $\underline{u}\leq \overline{u}$ and $\underline{v} \leq \overline{v} $ for all $t>0$.

\end{prop}
The proof is in exactly the same as in \cite{BRR1}, so we will not repeat it, it is essentially an application of the previous comparison principle, Proposition \ref{comparison}.

\section{Liouville-type result for general fields and invasion}\label{LiouvilleSection}

Our aim here is to prove the unicity of non-null stationary solutions of the system \eqref{syst}, with $\Omega$ a general field, satisfying Hypothesis \ref{gf}. The proof is based on an adaptation of the ``sliding method".

\subsection{Liouville-type result}

\begin{theorem}\label{Liouville}
Let $\Omega$ be a general field, i.e., satisfying Hypothesis \ref{gf} and consider the stationary system given by \eqref{syst} with field $\Omega$:

\begin{equation}\label{stat}
\left\{
\begin{array}{rllll}
-D\partial_{ss}^{2}u &=& \nu v-\mu u &\text{ for }& \quad t >0 \, , \, (x,y)\in \partial \Omega \\
-d\Delta v &=& f(v) &\text{ for }& \quad t >0 \, , \, (x,y)\in  \Omega \\
d\partial_{n}v &=& \mu u-\nu v &\text{ for }& \quad t >0 \, , \, (x,y)\in \partial \Omega.
\end{array}
\right.
\end{equation}
Then, $(\frac{\nu}{\mu},1)$ is the unique non-negative, non-null bounded solution.

\end{theorem}
Let us start with the following lemma :
\begin{lemma}
Let $(u,v)$ be a positive non-null bounded solution of \eqref{stat}.
Then,

\begin{equation}\label{eqinf}
\forall \epsilon>0 \quad , \quad \inf_{ dist(x,\partial\Omega)\geq \epsilon}v >0.
\end{equation}

\end{lemma}
\begin{proof}
Let $S$ be the set from Hypothesis \ref{gf}. Let $\epsilon > 0$. Without loss of generality, we will assume that $\epsilon < r$, where $r$ is the radius of the ball from the uniform ball condition. Take $x\in \Omega$ such that $dist(x,\partial \Omega) \geq \epsilon$. Then, there is $y\in \Omega$ such that $x\in B(y,r)\subset \Omega$ and $\vert x -y \vert=r-\epsilon$. Let us consider the isometry $\mathbb{I}$ such that $\mathbb{I}(S) \in \Omega$ and $\mathbb{I}(P)=y$. Let us call $\mathcal{E}:=\mathbb{I}(S) \cup B(y,r-\frac{\epsilon}{2}) \subset \Omega$ and let $E$ be a smooth set such that $\mathcal{E} \subset E \subset \mathbb{I}(S) \cup B(y,r-\epsilon)\subset \Omega$. We can choose this regularization independent of $y$.

Let us consider now the principal eigenfunction $\phi$ of $-\Delta$ on $E$ with Dirichlet conditions and the principal eigenvalue $\lambda$ for this problem, i.e., $-\Delta\phi=\lambda\phi$ on $E$, $\phi = 0$ on $\partial E $, $\phi>0$ in $E$. We normalize so that $\| \phi\|_{L^{\infty}}=1$. Because the principal eigenvalue is decreasing with respect to the inclusion, we have, by hypothesis, $\lambda \leq \lambda_{S} < \frac{f^{\prime}(0)}{d}$ ($\lambda_{S}$ is given by Hypothesis \ref{gf}). Now, observe that $\exists \eta_{0}$ such that $\forall s \in [0,\eta_{0}], f(s)\geq d\lambda s$. This implies that, if $\eta \in [0,\eta_{0}]$, $\eta \phi$ is a subsolution of $-d\Delta v = f(v)$ in $E$. Now, let 
\[
m := \inf \{\phi(x) \,, \, x\in E \,\text{ and } \, dist(x,\partial E) \geq \frac{\epsilon}{2} \}>0,
\]

which does not depend on $y$. Moreover, from the strong maximum principle, we have $v>0$ on $\Omega$. This means that there is $\eta_{1}>0$ such that for all $\eta < \eta_{1}$, $\eta \phi \leq v$ and for $\eta = \eta_{1}$, there is contact on $E$ between $v$ and $\eta_{1} \phi$ ($\eta_{1}$ is taken maximal). Therefore, $\eta_{1}\phi$ can not be a subsolution of the problem, because otherwise the strong maximum principle would imply  $v=\eta_{1}\phi$, which is impossible because of the Dirichlet condition on $\phi$. This means that we necessarily have $\eta_{1}>\eta_{0}$. Therefore, we have $\eta_{0}m \leq \eta_{1}\phi(x) \leq v(x)$. This does not depend on $y$ and so we have that:
\[
 \inf_{dist(x,\partial\Omega)>\epsilon}v \geq \eta_{0}m>0.
\]

\end{proof}
From the previous lemma, we immediately get:

\begin{cor}\label{conv1}
Let $(u,v)$ be a non-null solution of \eqref{stat}. Then
\[
\lim_{dist((x,y),\partial\Omega)\to +\infty}v(x,y)=1.
\]
\end{cor}

The proof is easy and comes from usual elliptic estimates and extractions. Let us now prove the Liouville-type result, Proposition \ref{Liouville}. The proof is divided into three steps. First, we show that $\inf_{\overline{\Omega}} v >0$, then, that the same holds for $u$ on $\partial \Omega$. To conclude, we use a sliding argument.

\begin{proof} 
\emph{First step.} Let us show that $\inf_{\overline{\Omega}} v>0$. Let us call $m:=\inf_{\overline{\Omega}} v$,
and let $(x_{k},y_{k})\in \Omega$ be a minimizing sequence for $v$. By contradiction, we assume that $m=0$. \\
From \eqref{eqinf}, up to extraction, we know that $dist((x_{k},y_{k}),\partial\Omega)\to 0 $.
For $k\in \mathbb{N}$, we consider $B_{k}\subset \Omega$ a ball tangent to the road with radius $r$ (we recall that $r$ is the radius of the ball from the uniform ball condition satisfied by $\Omega$) such that $(x_{k},y_{k})\in B_{k}$ (this is possible for $k$ large enough, so that $dist((x_{k},y_{k}),\partial \Omega)\leq r $). Call $(a_{k},b_{k})$ the center of this ball and consider the sequence of translated functions $v_{k}(x,y)=v(x+a_{k},y+b_{k})$. 
The functions $v_{k}$ are defined on $\overline{B}((0,0),r)$.
From standard interior and boundary elliptic estimates, there is a function $ v_{\infty} \geq 0$ such that $v_{k}\to v_{\infty}$ and such that there is a point $\xi \in \partial B((0,0),r)$ such that $v_{\infty}(\xi)=0$ and:
\[
\left\{
\begin{array}{lllr}
-d\Delta v_{\infty} &=& f(v_{\infty}) &\quad \text{on} \quad B((0,0),r)  \\
d\partial_{n}v_{\infty}(\xi) &\geq& 0,&
\end{array}
\right.
\]
where the second property comes from the fact that $\partial_{n}v_{k} \geq -\nu v_{k} $ on the translated boundary $\partial \Omega -\{(a_{k},b_{k})\}$ and the use partial boundary estimates, see lemma \cite[6.29]{GT}. The function $v_{\infty}$ is non-null because of \eqref{eqinf} and reaches a minimum at the boundary point $\xi$. Applying Hopf lemma at $\xi$ we get a contradiction.
Therefore, we have shown that:

\[
\inf_{\overline{\Omega}}v>0.
\]
\emph{Second step.} Let us show that $\inf_{\overline{\partial \Omega}} u>0$. We consider a minimizing sequence $(x_{k},y_{k})\in \partial \Omega$ and assume that the infimum of $u$ is zero, so that $u(x_{k},y_{k})\to 0$. 
Recall that the boundary $\partial \Omega$ is uniformly of class $C^{2,\alpha}$. This means that there is some $\lambda>0$ such that $\partial \big( \Omega - (x_{k},y_{k}) \big)\cap B((0,0),\lambda)$ is the graph of a $C^{2,\alpha}$ function. More precisely, from our hypotheses, we have the following : there is $\epsilon>0$ and there is a family of functions $\rho_{k}\in C^{2,\alpha}\left(\left] -\epsilon, \epsilon \right[\right)$, such that $\rho_{k}(0)=0$ and $\| \rho_{k} \|_{C^{2,\alpha}}$ is bounded independently of $k$. Moreover, there is a family of rotations $R_{k}$ such that

\[
R_{k}\Big( \partial \big( \Omega - (x_{k},y_{k}) \big)\cap B((0,0),\lambda) \Big) = \left\{ (s,\rho_{k}(s)) \,  , \, s \in \left] - \epsilon, \epsilon \right[ \right\}.
\]
Let  us define 
\[
\phi_{k}(s):=u(R^{-1}_{k}(s,\rho_{k}(s))+(x_{k},y_{k})) \, , \, s \in \left] - \epsilon, \epsilon \right[  .
\]
We have $\phi_{k}(0)=u(R^{-1}_{k}(0,0)+(x_{k},y_{k}))=u(x_{k},y_{k})$. Writing the Laplace-Beltrami in local coordinates, we see that $\phi_{k}(s)$ verifies the following uniformly elliptic equation on $\big] -\epsilon, \epsilon \big[$ :
\[
-a_{k}(s)\partial^{2}_{ss}\phi_{k}(s)+b_{k}(s)\partial_{s}\phi_{k}(s)+\mu \phi_{k}(s) \geq \nu \inf v,
\]
where $a_{k}=D\frac{1}{1+(\rho_{k}^{\prime})^{2}}$, $b_{k}=D\frac{\rho_{k}^{\prime}\rho_{k}^{\prime \prime}}{1+(\rho_{k}^{\prime})^{2}}$. Because of the hypotheses on the $\rho_{k}$, the families $(a_{k})_{k}$ and $(b_{k})_{k}$ are uniformly equicontinuous. Therefore, we can apply Ascoli-Arzela theorem on $\big[ -\frac{\epsilon}{2}, \frac{\epsilon}{2}\big]$ to get that, up to extraction, $a_{k}$ and $b_{k}$ converge uniformly to some functions $a_{\infty},b_{\infty}$. We can use usual interior elliptic estimates on $\phi_{k}$ and get that there is a function $\phi_{\infty}$ such that, for all $s\in \left]-\frac{\epsilon}{2},\frac{\epsilon}{2} \right[$ :

\[
-a_{\infty}(s)\partial^{2}_{ss}\phi_{\infty}(s)+b_{\infty}(s)\partial_{s}\phi_{\infty}(s)+\mu \phi_{\infty}(s) \geq \nu \inf v,
\]
with the property that there is a constant $c>0$ such that $a_{\infty}\geq c$ (this is obvious from the expression of $a_{k}$). Moreover, $\phi_{\infty}(0,0)=0$ is a minimum for $\phi_{\infty}$. Because $\nu \inf v>0$, this leads to a contradiction. Therefore, we have shown that:

\[
\inf_{\partial\Omega}u>0.
\]

\emph{Third step.} Let us now conclude the unicity. Let $(u_{1},v_{1})$ and $(u_{2},v_{2})$ be two bounded positive solutions of the stationary system. Then, define: 
\[
\theta:=\sup\{t>0/(u_{1},v_{1})>t(u_{2},v_{2})\}>0.
\]
Since those functions are bounded and since their infimum is strictly positive, $\theta$ is well defined.
By continuity $(u_{1},v_{1})\geq\theta(u_{1},v_{1})$. Let us show that $\theta\geq1$. If we do so, exchanging the roles of $(u_{1},v_{1})$ and $(u_{2},v_{2})$ will yield $(u_{1},v_{1})=(u_{2},v_{2})$. By contradiction, let us assume that $\theta<1$.

Let us define $U :=u_{1}-\theta u_{2}$ and $V :=v_{1}-\theta v_{2}$.
These functions are non-negative. By definition of $\theta$, there
is a sequence $(x_{k},y_{k})$ such that, either $U(x_{k},y_{k})\to 0$
or $V(x_{k},y_{k})\to 0$. To start, assume that $V(x_{k},y_{k})\to 0$. Because of Corollary \ref{conv1}, up to extraction, we can assume that $dist((x_{k},y_{k}),\partial \Omega)$ is bounded.

First, assume that there is $\epsilon>0$ such that $dist((x_{k},y_{k}),\partial \Omega)>\epsilon$. We can assume that $\epsilon<r$.  We will call $\delta$ the Euclidian distance in the plane.
As we did before, we can consider a sequence $(a_{k},b_{k})$ such that $(x_{k},y_{k})\in \overline{B}_{r}(a_{k},b_{k})$ and $\delta((x_{k},y_{k}),(a_{k},b_{k}))\leq r-\epsilon$. Up to extraction, we can assume that  $(x_{k}-a_{k},y_{k}-b_{k})$ converges in $\overline{B}_{r-\epsilon}(0,0)$ to some $(\tilde{x},\tilde{y})$. We can consider $v_{i,k}(x,y)=v_{i}(x+a_{k},y+b_{k})$, $i=1,2$, defined on $\overline{B}_{r}(0,0)$. Thanks to elliptic estimates, these functions converge $C^{2}_{loc}$ to $v_{i,\infty}$, $i=1,2$. These two functions are non-null because of the estimates on the infimum \eqref{eqinf}. Let $V_{\infty}=v_{1,\infty}-\theta v_{2,\infty}$. This function is positive, and satisfies: 
\[
-d\Delta V_{\infty}=f(v_{1,\infty})-\theta f(v_{2,\infty})>f(v_{1,\infty})- f(\theta v_{2,\infty}).
\]
This comes from the fact that $0<\theta<1$ and the KPP hypothesis. Therefore, there is a bounded function $b$ such that:
\[
-d\Delta V_{\infty}+bV_{\infty}>0.
\]
Moreover, $V_{\infty}(\tilde{x},\tilde{y})=0$. As $(\tilde{x},\tilde{y})\in B_{r}(0,0)$, we can apply the strong elliptic maximum principle at point $(\tilde{x},\tilde{y})$ to get a contradiction.

Assume now that $dist((x_{k},y_{k}),\partial \Omega)\to 0$. Using the same sequences $v_{i,k}$ and extracting again, we get that there is a function $V_{\infty}$ and $\xi \in \partial B_{r}(0,0)$ such that $V_{\infty}(\xi)=0$ and 
\[
\left\{
\begin{array}{rll}
-d\Delta V_{\infty}+bV_{\infty} &>0 &\quad \text{on} \quad B((0,0),r)  \\
d\partial_{n}V_{\infty}(\xi) &\geq 0.
\end{array}
\right.
\]
Again, the last equation comes from $d\partial_{n}(v_{1,k}-\theta v_{2,k})\geq - \nu(v_{1,k}-\theta v_{2,k})$ on $\partial \Omega$ and the use of portion boundary estimates. But using Hopf maximum principle at point $\xi$ leads to a contradiction. Then, we have shown that $\inf_{\overline{\Omega}} V>0$ on $\overline{\Omega}$.

Now, we are left to consider the case $U(x_{k},y_{k})\to 0$. But using the same kind of arguments as in the second step (i.e., unfolding the boundary), we will reach a contradiction. This yields the result.
\end{proof}

\subsection{Invasion}
 
 From what precedes, we can deduce that our system has the following invasion property:
 \begin{prop}\label{inv}
 Let $\Omega$ be a general field, i.e., satisfying Hypothesis \ref{gf}. Let $(u_{0},v_{0})$ be an initial non-negative non-null datum. Then, if $(u(t,x,y),v(t,x,y))$ is the solution of \eqref{syst} with field $\Omega$ arising from $(u_{0},v_{0})$, there holds:
 \[
 u(t,x,y) \underset{t \to +\infty}{\longrightarrow} \frac{\nu}{\mu} \quad , \quad v(t,x,y) \underset{t \to +\infty}{\longrightarrow}  1, 
\]
locally $C^{2,\alpha}$ in $(x,y)\in \partial \Omega$ (respectively in $\Omega$).

 \end{prop}
 The proof is very classical and will not be repeated here. The idea is to take a  small compactly supported subsolution and a large supersolution of the stationary problem as initial data for the parabolic problem and to show that the solutions arising are respectively time-increasing and time-decreasing and converge locally $C^{2}$ to a stationary solution. Using the previous Liouville-type result, one gets the result.

Now that we know that solutions arising from initially compactly supported data converge to an equilibrium, the question that arises is to find the speed of convergence. As said in the introduction, the geometry of the domain is here very important, and we will consider asymptotically conical fields, as defined in Hypothesis \ref{acf}. In particular, we will show that, whatever the angle $\theta_{0}$ of the cone, the speed of propagation in the direction of the road is the same as in the flat-road case of \cite{BRR1}. Such a result was not obvious \emph{a priori}, as will be explained in subsection \ref{subsecreq}.

\section{Exactly Conical-shaped field}\label{conical}
In this section, we study the spreading speed in the direction of the road when the field is \emph{exactly conical}, as defined after Hypothesis \ref{acf}. The next section will be dedicated to the case of \emph{asymptotically conical} fields.

Our previous results (the Liouville-type result and the (generalized) comparison principle) apply to this situation, as mentioned before. In the rest of this section, the half-opening angle $\theta_{0}$, the real number $a$ and the function $\rho$, given by Hypothesis \ref{acf}, are fixed.

Also, from now on, we shall sometimes use polar coordinates $(r,\theta)$. The angle $\theta$ we consider is the angle with respect to the axis of the cone, positive on the right, i.e., $(x,y)=(r\sin(\theta),r\cos(\theta))$ (observe that this is not the usual orientation for polar coordinates, but this comes more handy in our situation).

\subsection{Remarks on supersolutions}\label{subsecreq}

Before we start, let us quickly recall the strategy used in \cite{BRR1} to find the spreading speed in the direction of the road when the road is a line. The authors seek solutions of the linearized system (hence supersolutions of the system because of the KPP property) of the form $(u(t,x),v(t,x,y)) := (e^{-\alpha(x-ct)},\gamma e^{-\alpha(x-ct)}e^{-\beta y})$ where $\alpha,\gamma>0$ and $\beta\in \mathbb{R}$. Then, the system \eqref{syst} boils down to the following algebraic system:
\begin{equation}\label{explic}
\left\{
\begin{array}{rlcl}
-D\alpha^{2}+c\alpha &= \nu \gamma-\mu \\
-d\alpha^{2}+c\alpha &= f^{\prime}(0) + d \beta^{2}\\
d\beta \gamma &= \mu -\nu \gamma.
\end{array}
\right.
\end{equation}
Let us here say a word, informally, about how such a system is solved, because we shall use the same kind of arguments in the following. One can see that the third equation gives $\gamma = \frac{\mu}{\nu+d \beta}$. The first equation has the roots:
\[
\alpha^{\pm}(c,\beta)=\frac{1}{2D}\left( c \pm \sqrt{c^{2}+\frac{4\mu d D \beta}{\nu+d \beta}}\right).
\]
For this expression to be real, $\beta$ should be larger than some value we shall call $\tilde{\beta}(c)$. Working in the $(\beta , \alpha)$ plane, $\alpha^{\pm}$ is the graph of a curve we call $\Gamma(c)$. The second equation, in the $(\beta,\alpha)$ plane, is the equation of a circle of centre $(0,\frac{c}{2d})$ and of radius $\frac{\sqrt{c^{2}-c_{KPP}^{2}}}{2d}$. Hence, solving \eqref{explic} is equivalent to finding the intersections of the curve $\Gamma(c)$ and of this circle.This is done in details in \cite[Section 5]{BRR1}. In subsections \ref{supersol} and \ref{supersol2}, to find supersolutions, we reduce our problem to finding solutions of an algebraic system that will be \eqref{explic} but with some perturbations.

The critical speed $c_{BRR}$ is then defined as the smallest $c \geq 0$ such that system \eqref{explic} has a solution $(\alpha,\beta,\gamma)$ with $\alpha,\gamma>0$ (see \cite[section 5]{BRR1} for the details). Let us call $(\alpha^{\star},\beta^{\star},\gamma^{\star})$ the triple $(\alpha,\beta,\gamma)$ associated to $c_{BRR}$. It is shown in \cite{BRR1} that $(\alpha^{\star},\beta^{\star})>(0,0)$. One may want to try to use these supersolutions here in order to get an upper bound for the spreading speed. For notational simplicity, we can rotate the coordinates so that the right road is the half-axis $(y=0,x\geq0)$. An easy computation shows that, for $t$ large enough, $(u(t,x),v(t,x,y)):=(\min(\frac{\nu}{\mu},e^{-\alpha^{\star}(x-c_{BRR}t)}),\min(1,\gamma^{\star} e^{-\alpha^{\star}(x-c_{BRR}t)}e^{-\beta^{\star} y}))$ is a generalized supersolution of the system \eqref{syst} if and only if $2\theta_{0} \geq \pi - \arctan(\frac{\beta^{\star}}{\alpha^{\star}})$. Indeed, the level lines of $e^{-\alpha^{\star}(x-c_{BRR}t)}e^{-\beta^{\star} y}$ are the lines of equation $-\alpha^{\star}x-\beta^{\star} y=-\alpha^{\star}c_{BRR}t+C$, for $C\in \mathbb{R}$, and in the case where $2\theta_{0} \geq \pi - \arctan(\frac{\beta^{\star}}{\alpha^{\star}})$, this level lines do not cross the left road, if $-C$ is large enough. Up to a time translation, we can then assume that $(u(t,x,y),v(t,x,y))=(\frac{\nu}{\mu},1)$ for all $t\geq0$ on the left road. Hence, $(u,v)$ is indeed a generalized supersolution.

This means that, when the angle $\theta_{0}$ is large enough, the speed in the direction of the road is less than $c_{BRR}$. One can then show, using a family of subsolutions constructed in \cite{BRR1}, that the speed in the direction of the road is in fact exactly equal to $c_{BRR}$ (the lower bound for spreading speed is established in Section \ref{subsol}).

One could then expect to see an acceleration when the angle $\theta_{0}$ is small enough, i.e., when $a$ is large enough (an acceleration of the spreading was observed in \cite{RTV} in a system with two parallels roads). 

The main result of this paper is that this is not the case : the spreading speed in the direction of the road is always $c_{BRR}$, independently of $\theta_{0}$. Moreover, we show that the speed of invasion is lower than $c_{BRR}$ in all the directions in the field. In the next subsection, we study the easier case where $D\leq2d$ (in this case, the diffusion on the road is ``slow" compared to the diffusion on the field, and hence the effect of the road will not be observed). This will be the occasion to introduce radial supersolutions. In the following subsection, we find supersolutions of the system in the more interesting case $ D> 2d$. Then, we shall see that subsolutions of the system with flat road (built in \cite{BRR1}) are again subsolutions in some sense. As already mentioned, the difficulties here come from the geometry of the domain. The results of this section are used in the next Section \ref{almostconical}, where we study asymptotically conical fields.

\subsection{Case $D\leq 2d$}\label{Dpetit}
We start with considering the easier case $D\leq2d$. In this case, the presence of the road has no influence on the speed of invasion in the field. To show this, we use radial supersolutions.

\begin{prop}
Let $\Omega$ be an exactly conical field, i.e., satisfying Hypothesis \ref{acf} with $\rho(x)=a\vert x \vert$ for $\vert x \vert \geq 1$. Assume that $D \leq 2d$. Then, for all directions in the cone, the spreading speed is the KPP speed $c_{KPP}=2\sqrt{df^{\prime}(0)}$. This spreading speed is understood in the following sense : if $(u,v)$ is the solution of \eqref{syst} arising from a non-negative compactly supported non-null initial datum, we have:
\[
\forall c>c_{KPP},
\sup_{\substack{(x,y) \in \Omega \\  \vert(x,y)\vert \geq ct}} v(t,x,y) \underset{t\to +\infty}{\longrightarrow}   0 , \sup_{\substack{(x,y) \in \partial\Omega \\ \vert(x,y)\vert \geq ct}} u(t,x,y) \underset{t\to +\infty}{\longrightarrow}  0
\]
and
\[
\forall  c<c_{KPP},
\inf_{\substack{(x,y) \in \Omega \\  \vert(x,y)\vert \leq ct}} v(t,x,y) \underset{t\to +\infty}{\longrightarrow}  1 , \inf_{\substack{(x,y) \in \partial\Omega \\ \vert(x,y)\vert \leq ct}} u(t,x,y) \underset{t\to +\infty}{\longrightarrow}  1.
\]
\end{prop}

\begin{proof}
First of all, to prove the upper bound, define, for $\alpha,\gamma$ positive to be chosen after, $(u(t,r,\theta),v(t,r,\theta)):=(\min(\frac{\nu}{\mu},e^{-\alpha(r-ct)}),\min(1,\gamma e^{-\alpha(r-ct)}))$ where $(r,\theta)$ are the polar coordinates previously defined. For this couple to be a supersolution of \eqref{syst}, it is easy to see that it is sufficient to have:
\begin{equation*}
\left\{
\begin{array}{rlcr}
\alpha c -D\alpha^{2} &\geq \nu \gamma-\mu u  \\
\alpha c -d(\alpha^{2}-\frac{\alpha}{r}) &\geq f^{\prime}(0) & , & \forall r\geq 1 \\
0 &\geq \mu -\nu \gamma.
\end{array}
\right.
\end{equation*}

The condition $r \geq 1$ in the second equation comes from the fact that, up to a time translation, we can assume $(u,v)=(\frac{\nu}{\mu},1)$ if $r \leq 1$. Then, we can use the generalized comparison principle we proved earlier.
Let us take $\gamma = \frac{\mu}{\nu}$. The above algebraic system admits solutions $\alpha,c,\gamma>0$ if $c^{2}\geq4df^{\prime}(0)$ and if: 

\[
\frac{c+\sqrt{c^{2}-4df^{\prime}(0)}}{2d}\leq\frac{c}{D}.
\]

Because $D \leq2d$, we can choose $c=2\sqrt{df^{\prime}(0)}$, hence the first part of the result is given by comparison with $(u,v)$, indeed, if $\overline{c}>c$, we have : $\sup_{\substack{ \vert (x,y) \vert \geq \overline{c} t \\ (x,y) \in \partial \Omega }} u \underset{t \to +\infty}{\longrightarrow}  0$ and $\sup_{\substack{\vert (x,y) \vert \geq \overline{c} t \\ (x,y)\in \Omega}} v \underset{t \to +\infty}{\longrightarrow}~0 $.

For the lower bound, we proceed by contradiction. Let us take $\underline{c},\overline{c} \geq 0$, such that $\underline{c}<\overline{c}<2 \sqrt{df^{\prime}(0)}$. Assume that we have two sequences $t_{n},(x_{n},y_{n})$ such that $\vert (x_{n},y_{n}) \vert \leq \underline{c}t_{n} < \overline{c}t_{n}$ and $t_{n} \to +\infty $ and $v(t_{n},x_{n},y_{n})\to a$, for some $a<1$. We can define the sequences $v_{n}(t,x,y):=v(t+t_{n},x+x_{n},y+y_{n})$ and $u_{n}(t,x,y):=u(t+t_{n},x+x_{n},y+y_{n})$, defined respectively on $\Omega-(x_{n},y_{n})$ and $\partial\Omega-(x_{n},y_{n})$. Let us take $R>0$ large enough so that the following principal eigenvalue problem
\begin{equation*}
\left\{
\begin{array}{rlcl}
-\Delta \phi &= \lambda_{R}\phi \quad &\text{on} \quad B_{R}(0,0) \\
\phi &= 0  \quad &\text{on} \quad \partial B_{R}(0,0) \\
\phi&>0 \quad &\text{on} \quad B_{R}(0,0) \\
\| \phi \|_{L^{\infty}} &= 1.
\end{array}
\right.
\end{equation*}
has a solution with $ \lambda_{R}>0$ small enough, to be chosen after.

Now, for every $c\geq 0$ such that $c\leq \overline{c}$ and for $e=(e_{1},e_{2})\in \mathbb{S}^{1}$ in the cone, define:

 \[
 \left(\underline{u}_{c,e}(t,x,y),\underline{v}_{c,e}(t,x,y)\right):=\left(0,\epsilon \phi(x-cte_{1},y-cte_{2}-\tau) e^{-\frac{c}{2d}((x,y)\cdot e-ct)}\right),
 \]
 where $\tau \in \mathbb{R}^{+}$, depending only on $R$, is such that $B_{R}(0,\tau)\subset\Omega$ (hence $\underline{v}_{c,e}$ is well defined). Using the KPP hypotheses, it is classical to see that there is $\epsilon>0$, depending only on $f$ and $d$, such that these are subsolutions of the system. Now, taking $\epsilon$ smaller if needed, and because of the strong maximum principle that guaranties $v(1,x,y)>0$ for all $(x,y) \in \Omega$, we can assume that we have: 
 \[
 \epsilon \leq e^{-\frac{\overline{c}}{2d}(R+ \tau  )}\inf_{\vert (x,y)\vert \leq R +  \tau  }v(1,x,y).
 \]
This implies that, for all $c\leq \overline{c}$, for all $e$ in the cone: 
\[
(\underline{u}_{c,e}(1,x,y),\underline{v}_{c,e}(1,x,y))\leq(u(1,x,y),v(1,x,y)),
\]
and hence, by our maximum principle that: 
\[
(\underline{u}_{c,e}(t,x,y),\underline{v}_{c,e}(t,x,y))\leq(u(t,x,y),v(t,x,y)),
\]
for all $t\geq 1$. Therefore, for all $e$ in the cone, for all $c\leq \overline{c}$, we have 
\[
\epsilon \phi(x-cte_{1},y-cte_{2}-\tau) e^{-\frac{c}{2d}((x,y)\cdot e -ct)}\leq v(t,x,y).
\]
 This implies that $\inf_{\vert (x,y)\vert \leq \overline{c}t}v(t,x,y+\tau) \geq \epsilon$. Because $\underline{c}<\overline{c}$, we have that, for all $t>-t_{n}$, for all $(x,y) \in \Omega-(x_{n},y_{n})$, $v_{n}(t,x,y+\tau)\geq \epsilon$ if $n$ is large enough. Hence, using parabolic estimates, the sequence $v_{n}$ converge to an entire solution, name it $v_{\infty}$. It is solution on a half plane if $d((x_{n},y_{n}),\partial \Omega)$ is bounded or on a whole plane if not. Let us now show that $v_{\infty}=1$. We will write the end of the proof so that we do not need to differentiate those cases.

 There is a ball $B_{\delta} \subset \Omega$ such that, for all $t\in \mathbb{R}$, $v_{\infty}(t,x,y)\geq \epsilon$ on $B_{\delta}$. Therefore, if $m \in \mathbb{N}$, we have $v_{\infty}(-m,x,y)\geq \epsilon 1_{B_{\delta}}(x,y)$. If $(\tilde{u},\tilde{v})$ is the solution of \eqref{syst} arising from the initial datum $(0,\epsilon 1_{B_{\delta}}(x,y))$, we have that $v_{\infty}(0,x,y)\geq\tilde{v}(m,x,y)$ for all $(x,y) \in \Omega$. Now, letting $m\to \infty$ and using the invasion property Proposition \ref{inv}, we have that $v_{\infty}(0,x,y)=1$ for all $(x,y)\in \Omega$, which is impossible because $v_{\infty}(0,0)=a<1$. We have then reached a contradiction. Hence, if $c<c_{KPP}$, $\inf_{\substack{(x,y) \in \Omega \\  \vert(x,y)\vert \leq ct}} v(t,x,y) \underset{t\to \infty}{\longrightarrow}  1$. The same holds for $u$, to see that, it suffices to consider a sequence $t_{n}\to \infty$ and $(x_{n},y_{n})\in \partial \Omega$ such that $u(t_{n},x_{n},y_{n})\to \lim_{t\to \infty} \inf_{\vert (x,y)\vert\leq ct}u$. Considering $u_{n}$ and $v_{n}$ defined as before, we have that $(u_{n},v_{n}) \to (u_{\infty},1)$, this limiting couple being a solution of the system \eqref{syst} with $\Omega$ being a half-plane. Necessary, $u_{\infty}=\frac{\nu}{\mu}$, which means that $\inf_{\substack{(x,y) \in \partial\Omega \\  \vert(x,y)\vert \leq ct}} u(t,x,y) \underset{t\to \infty}{\longrightarrow}  \frac{\nu}{\mu}$, hence the result.
\end{proof}

\subsection{Case $D \geq 2d$. Supersolutions and upper bound for the speed in all directions}\label{supersol}

We show here that the speed of invasion is lower than $c_{BRR}$ in each direction in the field (we recall that $c_{BRR}$ is the speed in the direction of the road in the case where the road is a straight line, computed in \cite{BRR1}).

\begin{prop}\label{th1}
Let $\Omega$ be an exactly conical field, i.e., satisfying Hypothesis \ref{acf} with $\rho(x)=a\vert x \vert$ for $\vert x \vert \geq 1$. Let $(u,v)$ be the solution of the field-road system \eqref{syst} with field $\Omega$ arising from a  non-negative compactly supported initial datum $(u_{0},v_{0})$, then:
\[
 \forall c>c_{BRR}, \, \sup_{\substack{ \vert (x,y) \vert \geq c t \\ (x,y) \in \partial \Omega }} u(t,x,y) \underset{t \to +\infty}{\longrightarrow}  0 \quad , \quad  \sup_{\substack{\vert (x,y) \vert \geq c t \\ (x,y)\in \Omega}} v(t,x,y) \underset{t \to+ \infty}{\longrightarrow} 0. 
 \]

\end{prop}
If the road is a straight line, this estimation is optimal only in the direction of the road, c.f. \cite{BRR3}. The idea of the proof is to build almost radial supersolutions, as in the case $D\leq 2d$, whose levels sets are contained between two balls of radius growing as $ct$, for every $c>c_{BRR}$.

\begin{proof}
We will use again polar coordinates. Recall that our polar coordinates are $(r,\theta)\in \mathbb{R}^{+}\times [-\theta_{0},\theta_{0}]$, where $x=r\sin(\theta),y=r \cos(\theta)$. Let us fix $c>c_{BRR}$. We will construct supersolutions of the form: 

\[
\left\{
\begin{array}{l}
\overline{u}(t,r,\theta):=\min \left\{ \frac{\nu}{\mu}A,e^{-\alpha ( r-ct )} \right\}  \\
\overline{v}(t,r,\theta):=\min \left\{ A,\gamma \Psi (r,\theta) e^{-\alpha( r-ct )} \right\},
\end{array}
\right.
\]
where $A\geq\max\left\{1,\frac{\mu}{\nu}\sup u_{0}, \sup v_{0}\right\}$. We will assume that $f(v)$ is extended by zero for $v\geq 1$, which makes $(\frac{\nu}{\mu}A,A)$ a stationary solution of the system \eqref{syst}. The core of the proof is to build $\Psi$. It will be chosen even (with respect to $\theta$) and such that $\varepsilon \leq \Psi \leq 1$ for some $\varepsilon>0$ (the lower bound is necessary so that we could place compactly supported initial data under $\overline{v}$. The upper bound is just here for normalization). Observe also that, because of our generalized comparison principle Proposition \ref{gen}, we just need to check that $\left(e^{-\alpha ( r-ct )},\gamma \Psi (r,\theta) e^{-\alpha( r-ct )}\right)$ is a supersolution of the system \eqref{syst} for $(t,r,\theta)$ such that $e^{-\alpha ( r-ct )}\leq \frac{\nu}{\mu}$ or $\gamma \Psi (r,\theta) e^{-\alpha( r-ct )}\leq 1$. Up to translation, the points where this happens can be chosen far from the origin. Actually, it is sufficient to show that $\left(e^{-\alpha ( r-ct )},\gamma \Psi (r,\theta) e^{-\alpha( r-ct )}\right)$ is a supersolution of \eqref{syst} for $r\geq R$, for $R>0$ as large as we want. We will choose $\gamma,  \Psi,  \alpha $, with $\alpha>0, \gamma>0$ such that this is true.

Let us implement $\left(e^{-\alpha ( r-ct )},\gamma \Psi (r,\theta) e^{-\alpha( r-ct )}\right)$ in the system \eqref{syst} with field $\Omega$, for $r\geq R$, $R$ to be chosen after. Using the local representation of the Laplace-Beltrami operator and the expression of the gradient in polar coordinates, we see that these functions will be supersolutions if:

\begin{equation}\label{systpsi}
\left\{
\begin{array}{lllll}
\alpha c -D \alpha^{2} &\geq& \nu \gamma \Psi- \mu \quad &\text{on}& \quad \partial \Omega \cap\{R \leq r \} \\
\alpha c -d \frac{\tilde{\Delta }\Psi}{\Psi} &\geq& f^{\prime}(0) \quad &\text{on}& \quad  \Omega \cap\{R \leq r \}  \\
d\gamma \frac{1}{r} \partial_{\pm\theta}\Psi &\geq& \mu -\nu  \gamma \Psi \quad &\text{on}& \quad \partial \Omega \cap\{R \leq r \},
\end{array}
\right.
\end{equation}
where we write $\tilde{\Delta }\Psi :=e^{\alpha (r-ct)}\Delta e^{-\alpha (r-ct)}\Psi$.

Because the Laplace operator rewrites, in polar coordinates  $\partial^{2}_{rr}+\frac{1}{r}\partial_{r}+\frac{1}{r^2}\partial_{\theta\theta}$, we have:
\[
\tilde{\Delta }\Psi=\partial^{2}_{rr}\Psi-2\alpha\partial_{r}\Psi+\alpha^{2}\Psi+\frac{1}{r}\partial_{r}\Psi-\frac{\alpha}{r}\Psi+\frac{1}{r^2}\partial_{\theta\theta}\Psi.
\]
Now, we will state what conditions we want to impose on $\Psi$ and then see that these conditions imply \eqref{systpsi}, at least if $R$ is large enough. We will construct $\Psi$ just after.

Let us fix $0 < \eta <1$, $\alpha>0$ and $\beta \in \mathbb{R}$. We claim that there is $\Psi $ such that:

\begin{equation}
\left\{
\begin{array}{l}\label{condpsi}
\Psi \quad \text{is even w.r.t } \theta \\
\Psi(r,\pm\theta_{0})=1,\quad \forall r \geq 0 \\
0<  \frac{\eta}{1+\eta} \leq \Psi \leq 1 \\
\frac{\tilde{\Delta }\Psi}{\Psi} \leq \beta^{2} + \alpha^{2}  \, \text{ for } \,  r \geq R\\
\frac{1}{r} \partial_{\pm\theta}\Psi(r,\pm\theta_{0}) \geq \frac{\beta}{\eta+1}.
\end{array}
\right.
\end{equation}
If these conditions are satisfied, it is then sufficient for \eqref{systpsi} to hold that the following algebraic conditions on $\alpha,\beta,\gamma$ are verified:

\begin{equation}\label{perturbed}
\left\{
\begin{array}{l}
\alpha c -D \alpha^{2} \geq \nu \gamma - \mu\\
\alpha c  -d \alpha^{2} \geq f^{\prime}(0)+d\beta^{2}  \\
\frac{d \gamma \beta}{\eta+1}\geq \mu -\nu \gamma.
\end{array}
\right.
\end{equation}
However, as we mentioned in subsection \ref{subsecreq}, in \cite[section 5]{BRR1}, the authors showed that, when $c\geq c_{BRR}$,
\[
\left\{
\begin{array}{l}
\alpha c -D \alpha^{2} = \nu \gamma - \mu\\
\alpha c  -d \alpha^{2} = f^{\prime}(0)+d\beta^{2}   \\
d \gamma \beta = \mu -\nu \gamma
\end{array}
\right.
\]
has a solution $( \alpha^{\star},\beta^{\star})$, with $\alpha^{\star}>0$, $\beta^{\star}>0$. Using the same techniques we have that there is a solution $(\alpha,\beta)$ with $\alpha>0$, $\beta >0$ if $c>c_{BRR}$ and $ \eta $ is small enough. Moreover, $\epsilon>0$ being given, we can take $\eta$ small enough so that $\vert \alpha - \alpha^{\star}\vert \leq \epsilon$ and $\vert \beta - \beta^{\star} \vert \leq \epsilon$. Indeed, this comes from the fact that the algebraic problem \eqref{perturbed} is equivalent to finding the intersection in the $(\beta , \alpha)$ plane of 
\[
\alpha(\eta , c,\beta)^{\pm}:=\frac{c \pm \sqrt{c^{2}+\frac{4\mu D d\beta}{\nu(1+\eta)+d\beta}}}{2D}
\]
and of a circle of equation $(\alpha+\frac{c}{2d})^{2}+\beta^{2}=\frac{c^{2}-4df^{\prime}(0)}{4d^{2}}$. But when $\eta$ is small enough, the graph of $\alpha(\eta , c,\beta)^{\pm}$ is as close as we want to the graph of $\alpha(0,c,\beta)^{\pm}$ locally uniformly on $\beta$. This allow us to conclude. For the full details about the geometric interpretation of this system, see \cite[Section 5]{BRR1} where it is treated in details.

From now on, we suppose we have chosen $c>c_{BRR}$ and $(\alpha,\beta,\gamma,\eta)$ this way to be solutions of \eqref{perturbed}.
 Now, let us see how to construct $\Psi$ satisfying the conditions \eqref{condpsi}. We will take it of the form:
\[
\Psi(r,\theta)=\frac{1}{1+\eta}\left( \phi \left( \frac{\sqrt{R}}{\theta_{0}}(\vert \theta \vert - \theta_{0})+1 \right) e^{\beta r (\vert \theta \vert - \theta_{0})}+\eta \right),
\]
where $\phi$ is a $C^{\infty}$ function such that $\phi(x) = 0$ if $x<0$, $\phi(x) =1 $ if $x>1$, and $\phi$ strictly increasing in between. We can motivate our choice by observing that, when $\theta$ is close to $\theta_{0}$, then $e^{\beta r (\vert \theta \vert - \theta_{0})}$ behaves like $e^{\beta (x,y)\cdot n}$, where $n$ is the outer normal to the cone. this means that in some sense, up to rotation, near the road, our function $ \Psi (r,\theta) e^{-\alpha( r-ct )}$ behaves like the supersolution $e^{-\alpha(x-ct)}e^{-\beta y}$ of the flat-road case.

Let us now check that $\Psi$ meets the previously announced expectations. First, it is obvious that:
\[
\left\{
\begin{array}{l}
\Psi \quad \text{is even w.r.t } \theta \\
\Psi(r,\pm\theta_{0})=1,\quad \forall r \geq 0 \\
0< \frac{\eta}{1+\eta} \leq \Psi \leq 1. \\
\end{array}
\right.
\]
Let us compute $ \partial_{\theta}\Psi(r,\theta_{0}) $:
\[
\partial_{\theta}\Psi(r,\theta_{0}) = \frac{1}{1+\eta}\Big( \frac{\sqrt{R}}{\theta_{0}} \phi^{\prime}(1)+ \beta r \Big).
\]
Therefore, the condition
\[
\frac{1}{r} \partial_{\theta}\Psi(r,\theta_{0}) \geq \frac{\beta}{\eta+1}
\]
is verified.

We now need to check the last condition, i.e., for $r\geq R$:
\[
\partial^{2}_{rr}\Psi-2\alpha\partial_{r}\Psi+\alpha^{2}\Psi+\frac{1}{r}\partial_{r}\Psi-\frac{\alpha}{r}\Psi+\frac{1}{r^2}\partial_{\theta\theta}\Psi \leq (\beta^{2} + \alpha^{2} )\Psi.
\]
Because $\alpha>0$, $ \Psi \geq 0$ and $\partial_{r}\Psi\leq0$, it is sufficient to check: 

\[
\partial^{2}_{rr}\Psi-2\alpha\partial_{r}\Psi+\frac{1}{r^2}\partial_{\theta\theta}\Psi \leq \beta^{2} \Psi.
\]
This rewrites:
\[
\phi\beta^{2}(\vert \theta \vert -\theta_{0})^{2}-2\alpha\phi \beta (\vert \theta \vert - \theta_{0})+\frac{1}{r^2}\Big(\phi \beta^{2}r^{2} + 2\beta r \phi^{\prime}\frac{\sqrt{R}}{\theta_{0}}+\phi^{\prime\prime}\frac{R}{\theta_{0}^{2}}\Big) \leq \beta^{2} (\phi + \eta e^{-\beta r (\vert \theta \vert - \theta_{0})}).
\]
It is sufficient to have (remember that we assume $R \leq r$):
\[
 \vert \phi   (\vert \theta \vert - \theta_{0}) \vert \Big( \beta^{2}\theta_{0}+2\alpha \beta\Big)+\Big( 2\beta  \vert \phi^{\prime} \vert \frac{1}{\sqrt{R} \theta_{0}}+\vert \phi^{\prime\prime} \vert \frac{1}{R\theta_{0}^{2}}\Big) \leq \beta^{2} \eta e^{-\beta r (\vert \theta \vert - \theta_{0})}.
\]
Observe that $\beta^{2} \eta e^{-\beta r (\vert \theta \vert - \theta_{0})} \geq \beta^{2}\eta$.

Moreover, we can see that $ \phi \left( \frac{\sqrt{R}}{\theta_{0}}(\vert \theta \vert - \theta_{0})+1 \right)$ is equal to zero for $\vert \theta \vert \leq \theta_{0}(1-\frac{1}{\sqrt{R}})$. From that, it is easy to see that $\phi\left(\frac{\sqrt{R}}{\theta_{0}}(\vert \theta \vert - \theta_{0})+1\right)(\vert \theta \vert -\theta_{0})$ goes to zero as $R$ goes to infinity, uniformly in $\theta$. Therefore $ \phi  \vert (\vert \theta \vert - \theta_{0}) \vert ( \beta^{2}\theta_{0}+2\alpha \beta)\leq \frac{1}{2}\beta^{2}\eta  $ if $R$ is large enough (depending only on $\phi, \alpha^{\star},\beta^{\star} , \eta$). To conclude, we can take $R$ large enough such that:

\[
2\beta  \vert \phi^{\prime} \vert \frac{1}{\sqrt{R} \theta_{0}}+\vert \phi^{\prime\prime} \vert \frac{1}{R\theta_{0}^{2}} \leq \frac{1}{2} \beta^{2} \eta. 
\]
Then, the last condition is verified.

We have then built a (generalized) supersolution $(\overline{u},\overline{v})$ to our system. The result then follows : chose $A\geq 1$ such that $(u_{0},v_{0})< (A\frac{\nu}{\mu},A)$ is a compactly supported initial datum, then (up a translation in time for $(\overline{u},\overline{v})$), we have $(u_{0},v_{0}) \leq  (\overline{u},\overline{v})$, and the result follows by the generalized comparison theorem because $\sup_{\substack{ \vert (x,y) \vert \geq c t \\ (x,y) \in \partial \Omega }} \overline{u} \underset{t \to +\infty}{\longrightarrow}  0$ and $\sup_{\substack{\vert (x,y) \vert \geq c t \\ (x,y)\in \Omega}} \overline{v} \underset{t \to +\infty}{\longrightarrow}~0 $.
\end{proof}

Now, we investigate the existence of subsolutions, in order to get a lower bound for the spreading speed.

\subsection{Subsolutions and lower bound on the speed}\label{subsol}

The previous subsection \ref{supersol} gives us an upper bound for the spreading speed. In this subsection and in the next one, we give a lower bound. The standard way to do this is to find compactly supported stationary subsolutions moving at speed $c>0$, for all $c<c_{BRR}$, in the direction of the road.

In the paper \cite{BRR1}, where the spreading speed in the direction of the road was computed in the case where the road is a straight line, the author exhibited compactly supported non-negative subsolutions moving in the direction of the road. In this subsection, we shall see that we can use almost directly the same subsolutions. Therefore, we will not repeat their proof but just recall their result and see how it adapts. This will not work directly in the case where the field is asymptotically a cone, and we adress this issue in Section \ref{almostconical}, where we will need to deal then with geometric complications.

Let us recall the following result from \cite[Section 6]{BRR1}:
\begin{prop}\label{result}
Consider the field-road problem \eqref{syst} with $\Omega := \left\{ (x,y)\in \mathbb{R}^{2} \, , \, y\geq 0 \right\}$:
\begin{equation*}
\left\{
\begin{array}{rlcl}
\partial_{t}u-D\partial_{xx}^{2}u &= \nu v\vert_{y=0}-\mu u &\text{ for }& \quad t >0 \, , \, x\in \mathbb{R} \\
\partial_{t}v-d\Delta v &= f(v) &\text{ for }& \quad t >0 \, , \, (x,y)\in \mathbb{R}\times \mathbb{R}^{+}\\
-d\partial_{y}v\vert_{y=0} &= \mu u-\nu v\vert_{y=0}  &\text{ for }& \quad t >0 \, , \, x\in \mathbb{R}.
\end{array}
\right.
\end{equation*}
Assume $D>2d$. For $c<c_{BRR}$ close enough to $c_{BRR}$, there is a generalized subsolution of the form $\left(u(x-ct),v(x-cte,y)\right)$ where $(u,v)$ is non-null non-negative and compactly supported.
\end{prop}

First, let us make an observation. It is important to notice that the result holds only for $c<c_{BRR}$, $c$ ``close" to $c_{BRR}$. This is in contrast with subsection \ref{Dpetit} and leads to some complications. We deal with this difficulty in subsection \ref{strong spreading}.

The proof of Proposition \ref{result} is involved and will not repeated here. The main tool is a Rouché-type theorem. This implies directly the following:

\begin{prop}\label{sous}
Let $\Omega$ be an exactly conical field, i.e., satisfying Hypothesis \ref{acf} with $\rho(x)=a\vert x \vert$ for $\vert x \vert \geq 1$. Consider the field-road system \eqref{syst} and assume that $ D > 2d$. Then, if $e=(e_{1},e_{2}) := \frac{1}{\sqrt{1+a^{2}}}(1,a) \in \mathbb{S}^{1}$ is an unitary vector in the direction of the road, there are, for $c<c_{BRR}$ close enough to $c_{BRR}$, generalized subsolutions of the form $\left(u(x-cte_{1},y-cte_{2}),v(x-cte_{1},y-cte_{2})\right)$ where $(u,v)$ is non-null, non-negative and compactly supported.
\end{prop}
Take $(\underline{u}(x-ct,y),\underline{v}(x-ct,y))$ two compactly supported subsolutions of the system \eqref{syst} with $\rho=0$, as given by Proposition \ref{result}. Up to rotation of $\Omega$ of angle $\theta_{0}-\frac{\pi}{2}$, one just has to take $K$ large enough so that the support of $(\underline{u}(x-ct-K,y),\underline{v}(x-ct-K,y))$ is in $\Omega$ for $t\geq 0$, and then $(u(x-ct,y),v(x-ct,y)):=(\underline{u}(x-ct-K,y),\underline{v}(x-ct-K,y))$ is the moving subsolution we need (in the rotated coordinates).

\subsection{Spreading speed along the road}\label{strong spreading}
Now that we have supersolutions and subsolutions, we conclude the proof of Theorem \ref{mainth} in the case of an exactly conical field:

\begin{theorem}\label{thfinal}
Let $\Omega$ be an exactly conical field, i.e., satisfying Hypothesis \ref{acf} with $\rho(x)=a\vert x \vert$ for $\vert x \vert \geq 1$. Assume $ D > 2d$. Let $(u,v)$ be the solution of the system \eqref{syst} with field $\Omega$, with non-negative and compactly supported initial datum. Then:
\[
\forall c<c_{BRR},h>0, \quad \inf_{\substack{(x,y) \in \Omega \\ dist((x,y),\partial \Omega)<h \\ \vert(x,y)\vert \leq ct}} v(t,x,y) \underset{t\to +\infty}{\longrightarrow}  1, \inf_{\substack{(x,y) \in \partial \Omega \\  \vert(x,y)\vert \leq ct}} u(t,x,y) \underset{t\to +\infty}{\longrightarrow}  1,
\]
and
\[
\forall c>c_{BRR}, \quad \sup_{\substack{(x,y) \in \Omega \\  \vert(x,y)\vert \geq ct}} v(t,x,y)  \underset{t\to +\infty}{\longrightarrow} 0, \sup_{\substack{(x,y) \in \partial \Omega \\  \vert(x,y)\vert \geq ct}} u(t,x,y) \underset{t\to +\infty}{\longrightarrow}  0.
\]

\end{theorem}

The second property is exactly Proposition \ref{th1}. As already mentioned, the difficulty in proving the first property is that we only have subsolutions for $c<c_{BRR}$ close enough to $c_{BRR}$.

\begin{proof}
Let $(u,v)$ be the solution arising from a positive initially compactly supported datum $(u_{0},v_{0})$. Without loss of generality, we will assume that $ (u,v) \leq (\frac{\nu}{\mu},1)$. We want to prove that, for all $h>0$, $c<c_{BRR}$, 
\[
\inf_{\substack{(x,y) \in \Omega \\ dist((x,y),\partial \Omega)<h \\ \vert(x,y)\vert \leq ct}} v(t,x,y) \underset{t\to +\infty}{\longrightarrow}  1.
\]
Let $h>0$, $c<c_{BRR1}$, and consider sequences $t_{n},(x_{n},y_{n})$, with $t_{n}\to + \infty$, $dist((x_{n},y_{n}),\partial \Omega)<h$ and $\vert (x_{n},y_{n})\vert \leq ct_{n}$ and such that $v(t_{n},x_{n},y_{n})$ converges. Let us prove that:
\[
v(t_{n},x_{n},y_{n})\underset{n\to +\infty}{\longrightarrow}1.
\]
If $\vert x_{n} \vert \nrightarrow +\infty$, then $y_{n}$ is also bounded and, by the invasion property, $v(t_{n},x_{n},y_{n})\to~1$. Without loss of generality, we assume now that $x_{n}\to+\infty$.

Let $K\in \mathbb{R}$ and let $(\underline{u},\underline{v})$ be a compactly supported non-negative subsolution travelling in the direction $e$ of the right branch of the road, at some speed $c^{\prime} \in (c,c_{BRR})$, as given by Proposition \ref{sous}. Let us define $\tau_{n}=t_{n}- \frac{(x_{n},y_{n})\cdot e}{c^{\prime}}$.

As $(\underline{u}(0,x,y),\underline{v}(0,x,y))$ is compactly supported non-negative and strictly smaller than $1$, because $\tau_{n}\geq t_{n}(1-\frac{c}{c^{\prime}}) \to + \infty$ as $n \to + \infty$ and thanks to the invasion property, there is $n_{K}$ such that, if $n \geq n_{K}$, $ (\underline{u}(0,x,y),\underline{v}(0,x,y)) \leq ( u(\tau_{n}+K,x,y), v(\tau_{n}+K,x,y) )$. Thanks to the maximum principle Proposition \ref{gen}, we have that, for such $n$, for each $t\geq \tau_{n}+K $:
\begin{equation}\label{pdm}
(\underline{u}(t-\tau_{n}-K,\cdot,\cdot),\underline{v}(t-\tau_{n}-K,\cdot,\cdot) )\leq (u(t,\cdot,\cdot),v(t,\cdot,\cdot)).
\end{equation}
Let us define, for $(x,y)\in \Omega$ and $(x,y)\in \partial\Omega$ respectively, and for $  t \geq K+\tau_{n}$:
\begin{equation*}
\left\{
\begin{array}{l}
v^{n}_{K}(t,x,y)=\underline{v}(t-\tau_{n}-K,x,y) \\
u^{n}_{K}(t,x,y)=\underline{u}(t-\tau_{n}-K,x,y).
\end{array}
\right.
\end{equation*}
So that \eqref{pdm} rewrites, for $n\geq n_{K}$, for $t\geq \tau_{n}+K $:
\begin{equation}\label{appl}
(u^{n}_{K}(t,x,y),v^{n}_{K}(t,x,y) )\leq (u(t,x,y),v(t,x,y)).
\end{equation}
Moreover, from the construction of $(\underline{u},\underline{v})$ in \cite{BRR1}, we have that 
\[
(\underline{u}(t,(x,y)+c^{\prime}te),\underline{v}(t,(x,y)+c^{\prime}te))\geq(\epsilon,\epsilon),
\]
for some $\epsilon>0$ and for $\vert (x,y) \vert \leq \delta$ where $\delta>0$. This implies:

\begin{equation}\label{min}
(u^{n}_{K}(t_{n}+K,(x,y)+((x_{n},y_{n})\cdot e )e),v^{n}_{K}(t_{n}+K,(x,y)+((x_{n},y_{n})\cdot e )e) )\geq (\epsilon,\epsilon)
\end{equation}
for $\vert (x,y) \vert \leq \delta$.
Because $t_{n}+K\geq \tau_{n}+K$, we can apply \eqref{appl} at time $t=t_{n}+K$ to get, using \eqref{min}:
\[
\left( u(t_{n}+K,(x,y)+((x_{n},y_{n})\cdot e )e),v(t_{n}+K,(x,y)+((x_{n},y_{n})\cdot e )e)\right)\geq (\epsilon,\epsilon ),
\]
for $\vert(x,y)\vert \leq \delta$.

Now, define $u_{n}(t,x,y)=u(t+t_{n},(x,y)+((x_{n},y_{n})\cdot e )e)$ and $v_{n}(t,x,y)=v(t+t_{n},(x,y)+((x_{n},y_{n})\cdot e )e)$.

Thanks to the interior and partial boundary parabolic estimates, we can extract sequences that converge $C^{2,\alpha}_{loc}$, to entire solutions $(u_{\infty,}v_{\infty})$ of the system \eqref{syst}, where $\Omega$ is the graph of $\rho(x)=ax$.
 Moreover, because $(u_{n}(K,x,y),v_{n}(K,x,y))\geq (\epsilon,\epsilon)$ for $\vert (x,y) \vert \leq \delta$ and for all $K\geq -n $, we have that, for $\vert (x,y) \vert \leq \delta$ and for all $t$, $(u_{\infty}(t,x,y),v_{\infty}(t,x,y))\geq (\epsilon,\epsilon)$ .

Now, this implies that $(u_{\infty},v_{\infty})=(\frac{\nu}{\mu},1)$. Indeed, let $m>0$ and let $\epsilon>0$ small enough so that $\epsilon 1_{B_{\delta} } \leq v(-m,x,y)$. Let us take $(0,\epsilon 1_{B_{\delta}})$ as initial datum for the problem \eqref{syst} with $\rho(x)=ax$. The comparison principle Proposition \ref{gen} tells us that, if $(U,V)$ is the solution arising from this initial datum, we have:
\[
(U(m,x,y),V(m,x,y))\leq(u_{\infty}(0,x,y),v_{\infty}(0,x,y)).
\]
Because of the invasion property, letting $m\to + \infty$, we get
\[
(\frac{\nu}{\mu},1)\leq(u_{\infty}(0,x,y),v_{\infty}(0,x,y)).
\]
Observe that $(a_{n},b_{n}):=(x_{n}y_{n})-((x_{n},y_{n})\cdot e)e$ has bounded norm.
So, we have
\[
\lim_{n\to \infty} v(t_{n},x_{n},y_{n}) = \lim_{n\to \infty} v_{n}(0,a_{n},b_{n})\geq1.
\]
The last inequality comes from the $C^{2,\alpha}_{loc}$ convergence of $v_{n}$ to $v_{\infty}$. Because $v\leq 1$, this implies that $\lim_{n\to \infty} v(t_{n},x_{n},y_{n}) =1$. By the same token as in subsection \ref{Dpetit}, the same holds for $u$ : if we consider a sequence $t_{n}$ and a sequence $(x_{n},y_{n}) \in \partial \Omega$ such that $\vert (x_{n},y_{n}) \vert \leq ct_{n}$, we can define $(u_{n},v_{n}):=(u(\cdot + t_{n},\cdot + x_{n},\cdot + y_{n}),v(\cdot + t_{n},\cdot + x_{n},\cdot + y_{n}))$, which converges to some limit $(u_{\infty},1)$ solution of a system \eqref{syst} with field $\Omega$ a half plane. This implies that $u_{\infty}=\frac{\nu}{\mu}$, hence the result.
\end{proof}

Now, we turn to the more general case of a field that is asymptotically a cone.

\section{Asymptotically conical-shaped field}\label{almostconical}
We now generalize the results of the previous section to the case of a field that is asymptotically a cone, i.e.,  that satisfies Hypothesis \ref{acf}, proving then Theorem \ref{mainth}. Here again, we build supersolutions and subsolutions. The supersolutions we use are the same as in the previous section. The difficulty here is to build subsolutions.

\subsection{Supersolutions}\label{supersol2}

As before, we start to prove the following:

\begin{prop}\label{th2}
Let $\Omega$ be an asymptotically conical field, i.e., satisfying Hypothesis \ref{acf}. Let $c>c_{BRR}$. Therefore, if $(u,v)$ is the solution of the field-road system \eqref{syst} with field $\Omega$ arising from a  non-negative compactly supported initial datum, then:
\[
\sup_{\substack{ \vert (x,y) \vert \geq c t \\ (x,y) \in \partial \Omega }} u(t,x,y) \underset{t \to +\infty}{\longrightarrow}  0 \quad , \quad  \sup_{\substack{\vert (x,y) \vert \geq c t \\ (x,y)\in \Omega}} v(t,x,y) \underset{t \to +\infty}{\longrightarrow} 0.
\]

\end{prop}
The idea of the proof is to check that we can still use the same supersolutions as in subsection \ref{supersol}.

\begin{proof}
We shall use here again the same polar coordinates as in Section \ref{conical}, $(r,\theta)$, where $r\in (0,+\infty)$ and $\theta \in [ -\theta_{0},\theta_{0}]$. Again, we recall that we do not use the standard orientation, and that here we have $x=r \sin (\theta)$ and $y=r \cos(\theta)$.
Let us consider the functions, in polar coordinates $(r,\theta)$: 
\[
\left\{
\begin{array}{l}
\overline{u}(t,r,\theta):=\min \{ \frac{\nu}{\mu}A,e^{-\alpha ( r-ct )} \}  \\
\overline{v}(t,r,\theta):=\min \{ A,\gamma \Psi (r,\theta) e^{-\alpha( r-ct )} \},
\end{array}
\right.
\]
where $A\geq\max\left\{1,\frac{\mu}{\nu}\sup u_{0}, \sup v_{0}\right\}$. We will assume that $f(v)$ is extended by zero for $v\geq 1$, which makes $(\frac{\nu}{\mu}A,A)$ a stationary solution of the system \eqref{syst}. $\Psi$ was defined in the proof of Theorem \ref{th1}:
\[
\Psi(r,\theta)=\frac{1}{1+\eta}\left( \phi \left( \frac{\sqrt{R}}{\theta_{0}}(\vert \theta \vert - \theta_{0})+1 \right) e^{\beta r (\vert \theta \vert - \theta_{0})}+\eta \right).
\]
We recall that $\phi$ is a $C^{\infty}$ function such that $\phi(x) = 0$ if $x<0$, $\phi(x) =1 $ if $x>1$, and $\phi$ strictly increasing in between.

For notational simplicity, we introduce the functions, defined for $x\in \mathbb{R}$, $\tau(x)=\sqrt{1+\rho^{\prime 2}(x)}$, $\tilde{r}(x)=\sqrt{x^{2}+\rho^{2}(x)}$ and $\tilde{\theta}(x)$ such that, if $x\in \mathbb{R}$, $(\tilde{r}(x),\tilde{\theta}(x))$ is the representation in our polar coordinates of the boundary point, expressed in cartesian coordinates, $(x,\rho(x))$ (if $a>0$, then, at least for $x$ large enough, we would have $\tilde{\theta}(x)=\arctan(\frac{x}{\rho(x)})$). The first and third equations of the system are set on $\partial \Omega$. On this set, we will highlight the dependance of $\Psi$ and its derivatives on $x$ writing $\Psi = \Psi(\tilde{r}(x),\tilde{\theta}(x))$. Let $R$ be a positive number, to be chosen "large enough" after. As in the proof of Proposition \ref{th1}, we just have to plot $(e^{-\alpha ( r-ct ) },\gamma \Psi (r,\theta) e^{-\alpha( r-ct )})$ in the system \eqref{syst} and check that this is a supersolution, for $r \geq R$. An easy computation gives us that this couple is a supersolution if the following holds:

\begin{equation}\label{surR}
\left\{
\begin{array}{lll}
\alpha c -D \Big( \frac{\alpha^{2}}{\tau^{2}}(\tilde{r}^{\prime})^{2}-\frac{\alpha}{\tau} \left( \frac{\tilde{r}^{\prime}}{\tau}\right)^{\prime}    \Big) \geq \nu \gamma \Psi- \mu \quad &\text{on}& \quad \partial \Omega \cap B(0,R)^{c} \\
\alpha c -d \frac{\tilde{\Delta }\Psi}{\Psi} \geq f^{\prime}(0) \quad &\text{on}& \quad  \Omega \cap B(0,R)^{c} \\
\frac{d\gamma}{\tau}\frac{1}{\tilde{r}}\Big(  (\partial_{r}\Psi-\alpha\Psi)(x\rho^{\prime}-\rho)+\frac{1}{\tilde{r}}\partial_{\theta}\Psi(\rho \rho^{\prime}+x)\Big) \geq \mu -\nu  \gamma \Psi \quad &\text{on}& \quad \partial \Omega \cap B(0,R)^{c}.
\end{array}
\right.
\end{equation}
These come easily from the local coordinate representation of the Laplace-Beltrami~: $\partial_{ss}=\frac{1}{\tau}\partial_{x}(\frac{1}{\tau}\partial_{x})$ and changing coordinates. Observe that the equations on the road, at least when $R$ is large enough, are only perturbations of the system \eqref{perturbed}. To make this observation rigorous, we have to study the behavior of $\Psi,\partial_{r}\Psi, \partial_{\theta}\Psi,\tau$ and $\tilde{r}$ at infinity.

More precisely, denoting $o_{\infty}(z)$ a generic quantity that goes to $0$ as $z$ goes to $+\infty$, we want to show that : $\Psi(\tilde{r}(x),\tilde{\theta}(x))=1+o_{\infty}(\vert x \vert)$ and $ \partial_{r}\Psi(\tilde{r}(x),\tilde{\theta}(x))  =o_{\infty}(\vert x \vert)$ and also that $\frac{1}{\tilde{r}(x)}\partial_{\theta}\Psi(\tilde{r}(x),\tilde{\theta}(x)) = sign(\tilde{\theta}(x))\frac{1}{1+\eta}\beta+~o_{\infty}(\vert x \vert)$ (the quantity $sign(\tilde{\theta}(x))$ is equal to $1$ if $x$ is positive large enough and to $-1$ if $x$ is negative small enough).

This follows easily from the hypotheses. Indeed, we have that $\tilde{\theta}(x) \to \theta_{0}$. This implies that, for $\vert  x \vert $ large enough:
\[
\left\vert \tilde{r}(x)\left(\vert \tilde{\theta}(x)\vert -\theta_{0}\right)\right\vert \sim \sqrt{x^{2}+\rho^{2}(x)} \left\vert \sin(\vert \tilde{\theta}(x)\vert -\theta_{0}) \right\vert,
\]
and it is easy to see that $\sqrt{x^{2}+\rho^{2}(x)}\vert \sin(\vert \tilde{\theta}(x)\vert -\theta_{0}) \vert \leq \vert \rho(x)-a \vert x \vert \vert \to 0$. This implies that $\Psi(\tilde{r}(x),\tilde{\theta}(x))=1+o_{\infty}(\vert x \vert)$.

The fact that $ \partial_{r}\Psi(\tilde{r}(x),\tilde{\theta}(x))=o_{\infty}(\vert x \vert)$ can be proven the same way. The last fact, i.e., $\frac{1}{\tilde{r}(x)}\partial_{\theta}\Psi(\tilde{r}(x),\tilde{\theta}(x)) = sign(\tilde{\theta}(x))\frac{1}{1+\eta}\beta+~o_{\infty}(\vert x \vert)$ arises when combining what precedes and using that:

\begin{align*}
\frac{1}{\tilde{r}}\partial_{\theta}\Psi(\tilde{r},\tilde{\theta})=& \frac{1}{1+\eta}\Big(  sign(\tilde{\theta})\frac{\sqrt{R}}{\theta_{0}}\frac{1}{\tilde{r}}\phi^{\prime} \left( \frac{\sqrt{R}}{\theta_{0}}(\vert \tilde{\theta} \vert - \theta_{0})+1 \right) e^{\beta \tilde{r} (\vert \tilde{\theta} \vert - \theta_{0})}  \\
&+sign(\tilde{\theta})\beta \phi\left( \frac{\sqrt{R}}{\theta_{0}}(\vert \tilde{\theta} \vert - \theta_{0})+1 \right) e^{\beta \tilde{r} (\vert \tilde{\theta} \vert - \theta_{0})}  \Big).
\end{align*}
Now, some easy computations imply, if $(x,\rho(x))\in \partial \Omega \cap B(0,R)^{c}$:
 \[
 \Big( \frac{\alpha^{2}}{\tau^{2}}(\tilde{r}^{\prime})^{2}-\frac{\alpha}{\tau} \left( \frac{\tilde{r}^{\prime}}{\tau}\right)^{\prime}    \Big)=\alpha^{2}+o_{\infty}(\vert x \vert)
 \]
 and, using what precedes:
\[
\frac{1}{\tau \tilde{r}}\Big(  (\partial_{r}\Psi~-~\alpha\Psi)(x\rho^{\prime}-\rho)+~\frac{1}{\tilde{r}}\partial_{\theta}\Psi(\rho \rho^{\prime}+x)\Big)=  \frac{1}{1+\eta}\beta+o_{\infty}(\vert x \vert).
\]
Now, we can come back to the system \eqref{surR}. Chosing $R$ large enough, we have that the system is verified if the following algebraic system is verified, for $\epsilon>0$ as small as we want:
 
\begin{equation}
\left\{
\begin{array}{l}
\alpha c -D \alpha^{2}-D\epsilon \geq \nu \gamma (1+\epsilon) -\mu  \\
\alpha c  -d \alpha^{2} \geq f^{\prime}(0)+d\beta^{2}  \\
d\gamma\Big(  \frac{1}{1+\eta}\beta-\epsilon \Big) \geq \mu -\nu  \gamma (1-\epsilon).
\end{array}
\right.
\end{equation}
Once again, the presence of $\epsilon$ does not change the algebraic structure of the system~: as in subsection \ref{supersol}, we see that the problem is again equivalent to finding the intersection in the $(\beta , \alpha)$ plane of the graph of :
\[
\alpha(\eta , \epsilon , c,\beta)^{\pm} := \frac{c\pm \sqrt{c^{2}+4D\mu\frac{d\beta-d(\eta+1)\epsilon}{d\beta-d\epsilon(1+\eta)+\nu(1-\epsilon)(1+\eta)}-4D^{2}\epsilon}}{2D}
\]
and of a circle of equation $(\alpha+\frac{c}{2d})^{2}+\beta^{2}=\frac{c^{2}-4df^{\prime}(0)}{4d^{2}}$ . But again, when $\eta, \epsilon$ are small enough (and $\epsilon$ can be chosen as small as we want by taking $R$ large enough) the graph of $\alpha(\eta , \epsilon , c,\beta)^{\pm}$ is as close as we want to the graph of the $\alpha(0,0,c,\beta)^{\pm}$ locally uniformly on $\beta$. This allows us to conclude, using again the results of \cite[section 5]{BRR1}, that there are $\alpha>0$, $\gamma>0$ and $\beta>0$  such that the algebraic system is verified. Hence, $(\overline{u},\overline{v})$ is a generalized supersolution of our system, yielding the result by comparison.
\end{proof}

Now, we establish the equivalent of Proposition \ref{sous} for asymptotically conical fields, i.e., we find subsolutions moving at speed $c<c_{BRR}$ (for $c$ close enough to $c_{BRR}$) in the direction of the road.

\subsection{Subsolutions and conclusion}\label{lower bound}
In subsection \ref{subsol}, when the field was exactly conical, the particular geometry of the domain helped us : we were able to use the same subsolutions as in the flat-road case of \cite{BRR1}. Here, we have to deal with a more general domain, that looks asymptotically like a cone. Our aim in this subsection is to prove the same lower estimates on the spreading speed as in the previous section. 

\begin{prop}
Let $\Omega$ be an asymptotically conical field, i.e., satisfying Hypothesis \ref{acf}. Let $\rho$ be the function defining the boundary of $\Omega$ and $a$ the coefficient of the asymptotic to the right branch of the road. Consider the field-road system \eqref{syst} with field $\Omega$ and assume that $ D > 2d$. Then, if $e=(e_{1},e_{2}) := \frac{1}{\sqrt{1+a^{2}}}(1,a) \in \mathbb{S}^{1}$ is an unitary vector in the direction of the road, there are, for $c<c_{BRR}$ close enough to $c_{BRR}$, generalized subsolutions of the form $(u(x-cte_{1},y+\varepsilon(x)-cte_{2}),v(x-cte_{1},y+\varepsilon(x)-cte_{2}))$, where $(u,v)$ is nonnegative and compactly supported and where $\varepsilon(x) \to 0$ as $x \to +\infty$.
\end{prop}

The proof is divided in several steps. The idea is the following : first, changing variables, our problem will be set on an \emph{exactly} conical field instead of an \emph{asymptotically} conical field, but with coefficients that are non-constant. Then, following the same lines as in \cite{BRR1}, we build \emph{complex} subsolutions $(u,v)$ of the field-road system (actually, only near to the road) and build a couple of functions $(\phi,\psi)$ (not compactly supported) so that $(u+\phi,v+\psi)$ is subsolution of our problem. Finally, we turn this couple into a compactly supported generalized subsolution of the system by isolating one of its positivity component, in order to apply the generalized comparison principle Proposition \ref{gen}.

For simplicity, we focus on the case where $a=0$, i.e., the field is asymptotically a half plane. Once the result proven in this case, it extends directly to the case of a general coefficient $a$, by the same token as in the proof of Proposition \ref{sous} (rotation of the field and placing the moving subsolutions at time $t=0$ far from the origin).

\begin{proof}

\emph{First step : reduction to the case of an exactly conical field.} Consider the asymptotically conical field-road model \eqref{syst} with field $\Omega$. The function $\rho$, the coefficient $a$ and the angle $\theta_{0}$ are fixed. Let $(u_{0},v_{0})$ be a non-negative compactly supported initial datum and let $(u,v)$ be the solution arising from this datum. We define $\tilde{\rho}$ to be the graph of an \emph{exactly} conical field such that $\tilde{\rho}(x)= a \vert x \vert $ if $\vert x \vert \geq 1$. We call $\tilde{\Omega}$ the field given by the epigraph of $\tilde{\rho}$ (as defined in Hypothesis \ref{acf}). We define:

\begin{equation}\label{changvariables}
\left\{
\begin{array}{llllr}
\tilde{u}(t,x,\tilde{\rho}(x)) &:= &u(t,x,\rho(x)),   & \quad (x,\tilde{\rho}(x)) \in \partial \tilde{\Omega}               \\
\tilde{v}(t,x,y) &:= &v(t,x,y+\rho(x)-\tilde{\rho}(x)),  &  \quad (x,y)\in \tilde{\Omega},
\end{array}
\right.
\end{equation}
so that $(\tilde{u},\tilde{v})$ is defined on $\partial \tilde{\Omega} \times \tilde{\Omega}$. This couple of functions will satisfying a field road system but with non-constant coefficients. More precisely, define the diffusion matrix: 
\[
A(x)= \left( \begin{matrix}  1 & -(\rho^{\prime}-\tilde{\rho}^{\prime})(x) \\  -(\rho^{\prime}-\tilde{\rho}^{\prime})(x) & 1+(\rho^{\prime}-\tilde{\rho}^{\prime})^{2}(x)  \end{matrix} \right)
\]
and $\tau_{\rho}(x)=\sqrt{1+ \rho^{\prime 2}(x)}$, $\tau_{\tilde{\rho}}(x)=\sqrt{1+ \tilde{\rho}^{\prime 2}(x)}$. It is easy to see that the functions $(\tilde{u},\tilde{v})$ satisfy the following system (where $n$ is the outer normal to $\tilde{\Omega}$):

\begin{equation}
\left\{
\begin{array}{llcl}\label{changed variables}
\partial_{t}\tilde{u}-D\frac{1}{\tau_{\rho}}\partial_{x}(\frac{1}{\tau_{\rho}}\partial_{x}\tilde{u}) &= \nu \tilde{v}-\mu \tilde{u} &\text{ for }& \quad t >0 \, , \, (x,y)\in \partial \tilde{\Omega} \\
\partial_{t}\tilde{v}-d\nabla(A\nabla \tilde{v}) &= f(\tilde{v}) &\text{ for }& \quad t >0 \, , \, (x,y)\in \tilde{\Omega}\\
d \left(\frac{\tau_{\tilde{\rho}}}{\tau_{\rho}} \right) n A \nabla \tilde{v} &= \mu \tilde{u}-\nu \tilde{v} &\text{ for }& \quad t >0 \, , \, (x,y)\in \partial \tilde{\Omega}.
\end{array}
\right.
\end{equation}
This can be seen as a field-road system with general coefficients. What we will use here is that the coefficients become asymptotically constant, in some sense to be made clear after. We will build a compactly supported subsolution of this system \eqref{changed variables} moving in the direction of the road $\frac{1}{\sqrt{1+a^{2}}}(1,a):=(e_{1},e_{2})$ at speed $c<c_{BRR}$ of the form $\left(\Phi(x-cte_{1},\tilde{\rho}(x)-cte_{2}),\Psi(x-cte_{1},y-cte_{2})\right)$. Doing the inverse change of variables we did in \eqref{changvariables}, we will have a subsolution of the field-road system \eqref{syst} with field $\Omega$ of the form:
\[
\left(\Phi(x-cte_{1},\rho(x)+(\tilde{\rho}(x)-\rho(x))-cte_{2}),\Psi(x-cte_{1},y+(\tilde{\rho}(x)-\rho(x))-cte_{2})\right).
\]
Because of the hypotheses on $\rho, \tilde{\rho}$, we indeed have that $\vert \rho(x)-\tilde{\rho}(x)\vert \to 0$ as $\vert x \vert \to +\infty$, which gives the result.

As said before the proof, we focus on the case where $a=0$. In this case, we take $\tilde{\rho}(x)=0$. Then, $\tau_{\tilde{\rho}}=1$. For notational simplicity, from now on we forget the dependance of $\tau_{\rho}$ on $\rho$ and write $\tau$ instead. The system \eqref{changed variables} then rewrites:

\begin{equation}
\left\{
\begin{array}{llcl}\label{AsPlat}
\partial_{t}\tilde{u}-D\frac{1}{\tau}\partial_{x}(\frac{1}{\tau}\partial_{x}\tilde{u}) &= \nu \tilde{v}\vert_{\partial \Omega}-\mu \tilde{u} &\text{ for }& \quad t >0 \, , \, x\in \mathbb{R}\\
\partial_{t}\tilde{v}-d\nabla(A\nabla \tilde{v}) &= f(\tilde{v}) &\text{ for }& \quad t >0 \, , \, (x,y)\in \mathbb{R}\times\mathbb{R}^{+} \\
-d \frac{1}{\tau}  e_{y} A \nabla \tilde{v}\vert_{\partial \Omega} &= \mu \tilde{u}-\nu \tilde{v}\vert_{\partial \Omega} &\text{ for }&  \quad t >0 \, , \, x\in \mathbb{R}.
\end{array}
\right.
\end{equation}
\emph{Second step : Building complex subsolutions.} We start to introduce some notations. if $A$ is a $C^{1}$ matrix and if $\tau$ is a $C^{1}$ positive function, we define the three following linear operators: 
\begin{equation*}
\left\{
\begin{array}{lll}
 P^{1}_{\tau}\tilde{u} &:=&\partial_{t}\tilde{u}-D\frac{1}{\tau}\partial_{x}(\frac{1}{\tau}\partial_{x}\tilde{u}) \\
 P^{2}_{A}\tilde{v} &:=& \partial_{t}\tilde{v}-d\nabla(A\nabla \tilde{v}) \\
 B_{A,\tau}\tilde{v} &:=& -d \frac{1}{\tau} e_{y} A \nabla \tilde{v}\vert_{\partial \Omega}.
\end{array}
\right.
\end{equation*}
The system \eqref{AsPlat} rewrites under the more compact form:

\begin{equation*}
\left\{
\begin{array}{llcl}
P^{1}_{\tau}\tilde{u} &= \nu \tilde{v}\vert_{\partial \Omega}-\mu \tilde{u} &\text{ for }& \quad t >0 \, , \, x\in \mathbb{R}\\
P^{2}_{A}\tilde{v} &= f(\tilde{v}) &\text{ for }& \quad t >0 \, , \, (x,y)\in \mathbb{R}\times\mathbb{R}^{+} \\
B_{A,\tau}\tilde{v} &= \mu \tilde{u}-\nu \tilde{v}\vert_{\partial \Omega} &\text{ for }& \quad t >0 \, , \, x\in \mathbb{R}.
\end{array}
\right.
\end{equation*}
Let $\phi$ be a non-negative function, $C^{2}$ on its support, such that:
\begin{equation*}
\left\{
\begin{array}{l}
\phi \text{ is compactly supported.} \\
\phi(0)=0 \\
1 \leq d\phi^{\prime \prime}+f^{\prime}(0)\phi \\
\phi^{\prime}(0) = \frac{2\mu+1}{d}.
\end{array}
\right.
\end{equation*}
Let us take $M>0$ so that $supp(\phi) \subset [0,M]$. Let us take $L>M$ large enough to be chosen after. Now, let us explain how to build subsolutions of \eqref{AsPlat}. Because we are dealing with a KPP nonlinearity, it is a classical observation that it is sufficient to build subsolutions of the system \eqref{AsPlat} \emph{linearized and penalized}, i.e., replacing $f(v)$ by $(f^{\prime}(0)-\delta)v$, where $\delta>0$ is a small parameter. For the sake of clarity, we will omit the $\delta$ in the following, this can be done without loss of generality.

For $\lambda>0$, for $\alpha,\beta,\gamma_{1},\gamma_{2}$ complex numbers to be be chosen after, we define $\gamma(y):=\gamma_{1}e^{-\beta y}+~\gamma_{2}e^{\beta y}$, and also:
\[
(u_{\lambda},v_{\lambda}):=\Real\Big((1,\gamma(y))e^{-\alpha(x-ct)}\Big)+(-\lambda,\lambda \phi (y)).
\]
This couple is defined on $\partial \tilde \Omega \times \tilde{\Omega}$. We denote:
\[
(U(t,x),V(t,x,y)) := \Real\Big((1,\gamma(y))e^{-\alpha(x-ct)}\Big).
\]
We will now choose the coefficients $(\alpha,\beta,\gamma_{1},\gamma_{2})$ such that $(u_{\lambda},v_{\lambda})$ are subsolutions of \eqref{AsPlat}. We impose $v_{\lambda}(t,x,L)=0$, this implies (remember that $supp(\phi) \subset [0,M]$ and $L>M$):
\[
\gamma_{1}e^{-\beta L}+\gamma_{2}e^{\beta L} =0.
\]
The idea is to use $(u_{\lambda},v_{\lambda})$ to get subsolutions of \eqref{AsPlat}. It can be seen that our function can not be subsolution of the system in all the upper half-plane, therefore, we will will work on smaller domains $E,F$. Using the oscillating nature of $(U,V)$, we will be able to take $E,F$ such that they contain at least one positivity component of $(U,V)$. The generalized comparison principle \ref{gen} will allow us to prolongate these functions by zero outside of their positivity component. This will be made clear after, from now on, we consider the system \eqref{AsPlat} restricted to subdomains $E,F$ of $\mathbb{R},\mathbb{R}\times\mathbb{R}^{+}$, to be chosen after.

\emph{Third step : $(u_{\lambda},v_{\lambda})$ is subsolution.} The couple $(u_{\lambda},v_{\lambda})$ is subsolution of the linearized system on $E,F$ if:

\begin{equation}\label{restrict}
\left\{
\begin{array}{llcl}
P^{1}_{\tau}U &\leq \nu V-\mu U+\mu \lambda &\text{ for }& \quad t >0 \, , \, x\in E\\
P^{2}_{A}V -d\nabla(A\nabla \phi) &\leq f^{\prime}(0)V+\lambda f^{\prime}(0) \phi &\text{ for }& \quad t >0 \, , \, (x,y)\in F \\
B_{A,\tau}V+\lambda B_{A,\tau}\phi &\leq \mu U-\nu V-\mu \lambda &\text{ for }& \quad t >0 \, , \, x \in E.
\end{array}
\right.
\end{equation}
Equivalently :

\begin{equation*}
\left\{
\begin{array}{llcl}
(P^{1}_{\tau}-P^{1}_{0})U +P^{1}_{0}U&\leq \nu V-\mu U+\mu \lambda &\text{ for }& \quad t >0 \, , \, x\in E\\
(P^{2}_{A}-P^{2}_{I})V+P^{2}_{I}V -d \lambda \nabla(A\nabla \phi) &\leq f^{\prime}(0)V+\lambda f^{\prime}(0) \phi &\text{ for }& \quad t >0 \, , \, (x,y)\in F \\
(B_{A,\tau}-B_{I,0})V+B_{I,0}V+\lambda B_{A,\tau}\phi &\leq \mu U-\nu V-\mu \lambda &\text{ for }& \quad t >0 \, , \, x \in E.
\end{array}
\right.
\end{equation*}
Let us call
\[
\begin{array}{ll}
z_{1} =&  \alpha c - D\alpha^{2}-\nu\gamma(0)+\mu   \\
z_{2} =& \alpha c -d(\alpha^{2}+\beta^{2})-f^{\prime}(0) \\
z_{3} =& d\beta(-\gamma_{1}+\gamma_{2})-\mu + \nu (\gamma_{1}+\gamma_{2})) \\
z_{4} = &\gamma_{1}e^{-\beta L}+\gamma_{2}e^{\beta L} .
\end{array}
\]
Then, some computations show us that $P^{1}_{0}U -\nu V(\cdot,0) +\mu U = z_{1}U$, $P^{2}_{I}V-f^{\prime}(0)V=z_{2}V $ and $B_{I,0}V(\cdot,0)-\mu U+\nu V(\cdot,0)=z_{3}U$.

At this point, we need to recall some elements from \cite{BRR1}. See the proof of their theorem 6.1 for details. There, they show that :
 $\exists L_{0}>0$ such that $\forall L \geq L_{0}$, $\exists! c_{L} \in (c_{KPP}, c_{BRR})$, such that $c_{L}\to c_{BRR}$ as $L\to +\infty$ and such that : $\forall c < c_{L}$, close enough to $c_{L}$, there are $(\alpha,\beta,\gamma_{1},\gamma_{2})$ such that $z_{1}=z_{2}=z_{3}=z_{4}=0$. Moreover, we have:
 
\begin{equation}\label{posit}
\left\{
\begin{array}{c}
\Real(\beta) > 0 \\
\Imag(\beta) \neq 0 \\
\Imag(\alpha) \neq 0.
\end{array}
\right.
\end{equation}
They also show that, taking $c$ close enough to $c_{L}$, we can select two positivity component of $U$ and $V$, name them $E$ and $F$ respectively, such that the functions $(\tilde{U},\tilde{V})$, defined to be equal to $(U,V)$ on $E$ and $F$ and zero everywhere else are subsolutions in the sense of the generalized comparison principle \ref{gen}.

Moreover, if $\alpha_{1}, \beta_{1}$ are the real parts of $\alpha,\beta$ and  $\alpha_{2}, \beta_{2}$ the imaginary parts of $\alpha,\beta$, we have:
\begin{equation*}
\left\{
\begin{array}{lll}
U(t,x)&=& e^{-\alpha_{1}(x-ct)}\cos(\alpha_{2}(x-ct)) \\
V(t,x,y)&=&\vert \gamma_{1}\vert  e^{-\alpha_{1}(x-ct)}\Big( e^{-\beta_{1}y}cos(\arg(\gamma_{1}) -\alpha_{2}x-\beta_{2}y) \\ 
&& - e^{-\beta_{1}(2L-y)}cos(\arg(\gamma_{1}) -\alpha_{2}x-\beta_{2}(2L-y) ) \Big).
\end{array}
\right.
\end{equation*}
Observe that these functions are periodic in the $x$ direction, with period $\frac{2\pi}{\Imag(\alpha)}$.

Then, taking $(\alpha,\beta,\gamma_{1},\gamma_{2})$ such that $z_{1}=z_{2}=z_{3}=z_{4}=0$, $(u_{\lambda},v_{\lambda})$ is subsolution of \eqref{restrict} if:

\begin{equation}\label{perturbation}
\left\{
\begin{array}{llcl}
(P^{1}_{\tau}-P^{1}_{0})U &\leq \mu \lambda &\text{ for }&  t >0 \, , \, x\in E\\
(P^{2}_{A}-P^{2}_{I})V  &\leq \lambda f^{\prime}(0) \phi +d \lambda \nabla(A\nabla \phi) &\text{ for }&  t >0 \, , \, (x,y)\in F \\
(B_{A,\tau}-B_{I,0})V+\lambda B_{A,\tau}\phi &\leq-\mu \lambda &\text{ for }&  t >0 \, , \, x \in E.
\end{array}
\right.
\end{equation}
We need now to define $E$ and $F$ and to explain how we will truncate $(u_{\lambda},v_{\lambda})$ to get a compactly supported subsolution of \eqref{AsPlat}.

\emph{Fourth step : localisation and truncating.} For $A>0$ to be chosen after, define:
\[
F=\Omega_{L,\Lambda,R,c}:=\Omega \cap \{ x \geq \Lambda \} \cap \{y \leq  L\}\cap\{ \vert x -ct \vert \leq R\}
\]
and 
\[
E:=\overline{\Omega_{L,\Lambda,R,c}}\cap\{y=0\}.
\]
We can take $R$ large enough, depending on $U,V$ and $L$ so that $\bar{\Omega}_{L,\Lambda,R,c}$ contains at least one positivity component of $(U,V)$. On these sets, the $W^{2,\infty}$ norms of $(U,V)$ are uniformly bounded (by quantities depending on $\alpha,\beta,\gamma_{1},\gamma_{2},c$), and the quantities $\sup_{E}(P^{1}_{\tau}-P^{1}_{0})U$, $\sup_{F}(P^{2}_{A}-P^{2}_{I})V$, $\sup_{E}(B_{A,\tau}-B_{I,0})V$ go to zero as $\Lambda$ becomes large (because of the hypotheses on the convergence of $\rho$). Observe that, thanks to the hypotheses on $\phi$, if $x\geq \Lambda$, for $\Lambda$ large enough, $  f^{\prime}(0) \phi +d  \nabla(A\nabla \phi) \geq 1$. Then, the system \eqref{perturbation} is verified for $(x,y) \in \overline{\Omega_{L,\Lambda,R,c}}$ if:

\begin{equation*}
\left\{
\begin{array}{llcl}
\epsilon_{1}(\Lambda) &\leq \mu \lambda  \\
\epsilon_{2}(\Lambda)  &\leq \lambda \\
\epsilon_{13}(\Lambda) &\leq-\mu \lambda +d \lambda \tau(x) \phi^{\prime}(0) \quad ,\quad  \text{ for } x\geq \Lambda,
\end{array}
\right.
\end{equation*}
where $\epsilon_{1}(\Lambda),\epsilon_{2}(\Lambda),\epsilon_{3}(\Lambda)$ are functions that go to zero as $\Lambda$ goes to infinity. Then, we can always take $\Lambda$ large enough, depending on $\lambda$, so that the hypotheses on $\phi$ yields that these equations are verified.

\emph{Last step : conclusion.} Let us sum up what we have done. We have proven that, for each $\lambda>0$, the functions $(u_{\lambda},v_{\lambda})=(U-\lambda,V+\lambda\phi)$ are subsolutions of the system \eqref{restrict}, for $L, \Lambda$ large enough.

We saw that we can take $R$ large enough so that $(U,V)$ have at least one positivity component in this set. This $R$ depends only on $U,V$ and $L$,and not on $\Lambda$. Then, $L$ and $R$ being chosen, we can then take $\lambda$ small enough so that $(U-\lambda,V+\lambda \phi)$ has bounded positivity components on $\Omega_{L,\Lambda,R,c}$. Indeed, remember that $(u_{\lambda},v_{\lambda})=(U,V)+(-\lambda,\lambda\phi)$. Taking $\lambda$ small enough, the positivity component of $u_{\lambda}$ and $v_{\lambda}$ will be as close as we want to the positivity component of $U$ and $V$, respectively, in $(E,F)$. This is obvious for $u_{\lambda}(t,x)=e^{-\alpha_{1}(x-ct)}\cos(\alpha_{2}(x-ct))-\lambda$, for $x\in E$. For $v_{\lambda}$, observe first that $L$ was chosen large enough so that $L>M$, where $M$ is such that $supp(\phi)\subset [0,M]$. If $(x,y)\in \Omega_{L,\Lambda,R,c}$ is such that $\cos(\arg(\gamma_{1}) -\alpha_{2}x-\beta_{2}y)=-1$ and $y\leq M$, we have 
\[
V(t,x,y)\leq \vert \gamma_{1}\vert  e^{-\alpha_{1}(x-ct)}( -e^{-\beta_{1}y} + e^{-\beta_{1}(2L-y)}).
\]
Because $(x,y)\in \Omega_{L,\Lambda,R,c}$ and $y\leq M$, we have :
\[
V(t,x,y)\leq \vert \gamma_{1}\vert  e^{-\alpha_{1}R}e^{-\beta_{1}M}( e^{-2\beta_{1}(L-M)}- 1).
\]
Because $M<L$, this quantity is strictly negative ($\beta_{1}$ is positive, from \eqref{posit} ). Therefore, if $\lambda>0$ is small enough, $V-\lambda\phi$ has bounded positivity component on $\Omega_{L,\Lambda,R,c}$ as close as we want to those of $V$.

Now, if we define $(u,v)$ to be equal to $(u_{\lambda},v_{\lambda})$ on the positivity components we chose before, and prolongate these by zero in the rest of the upper half-plane, we get functions that are generalized subsolutions in the sense of the generalized comparison principle \ref{gen} of the system \eqref{restrict}. These functions are non-negative, compactly supported and move at speed $c$ in the direction $e$. Now, doing the inverse change of variable we did in \eqref{changvariables}, we get a couple of generalized subsolutions in the sense of the generalized comparison principle, Proposition \ref{gen}, of the system \eqref{syst}, hence the result.

\end{proof}

At this point, we have built supersolutions and subsolutions. Using the same strategy as in subsection \ref{strong spreading}, we can prove our main result, Theorem \ref{mainth}. The main difference is that the subsolutions we have are of the form $(u(x-cte_{1},y+\varepsilon(x)-cte_{2}),v(x-cte_{1},y+\varepsilon(x)-cte_{2}))$, where $\varepsilon(x)$ goes to zero as $x$ goes to infinity. This difference rises no difficulties, the proof of Theorem \ref{thfinal} adapts directly.

\noindent {\textbf{Acknoledgements:} The research leading to these results has received funding from the European Research Council under the European Union?s Seventh Framework Program (FP/2007-2013) / ERC Grant Agreement n.321186 - ReaDi - Reaction-Diffusion Equations, Propagation and Modeling. The author wants to thank Henri Berestycki and Luca Rossi for suggesting this problem and for interesting discussions.}

\end{document}